\documentclass[a4paper,twoside]{article}
\usepackage{a4}
\usepackage{amssymb}
\usepackage{amsmath}
\usepackage{url}
\usepackage[active]{srcltx}
\usepackage[pagebackref,colorlinks,citecolor=blue,linkcolor=blue,urlcolor=blue]{hyperref}
\usepackage[dvipsnames]{color}
\usepackage{upref}
\usepackage{enumitem}
\usepackage{marginnote}
\allowdisplaybreaks[2] 
\definecolor{gruen}{cmyk}{1.0,0.2,0.7,0.07}
\definecolor{greenvariant}{cmyk}{1.0,0.1,0.4,0.1}
\definecolor{mag}{cmyk}{0.0,0.9,0.3,0.0}

\newcounter{komcounter}
\numberwithin{komcounter}{section}

\newcount\minutes \newcount\hours
\hours=\time
\divide\hours 60
\minutes=\hours
\multiply\minutes -60
\advance\minutes \time
\newcommand{\klockan}{\the\hours:{\ifnum\minutes<10 0\fi}\the\minutes}
\newcommand{\tid}{\today\ \klockan}
\newcommand{\prtid}{\smash{\raise 10mm \hbox{\LaTeX ed \tid}}}
\makeatletter
\def\sectionmark#1{} 
\def\subsectionmark#1{}
\newcommand{\sectnr}{\ifnum \c@secnumdepth >\z@
                 \thesection.\hskip 1em\relax \fi}
\def\@evenhead{\footnotesize\rm\thepage\hfil\leftmark\hfil\llap{\prtid}}
\def\@oddhead{\footnotesize\rm\rlap{\prtid}\hfil\rightmark\hfil\thepage}
\def\tableofcontents{\section*{Contents} 
 \@starttoc{toc}}
\makeatother
\makeatletter
\def\@biblabel#1{#1.}
\makeatother
\makeatletter
\let\Thebibliography=\thebibliography
\renewcommand{\thebibliography}[1]{\def\@mkboth##1##2{}\Thebibliography{#1}
\addcontentsline{toc}{section}{References}
\frenchspacing 
\setlength{\@topsep}{0pt}
\setlength{\itemsep}{0pt}%
\setlength{\parskip}{0pt plus 2pt}%
}
\makeatother
\makeatletter
\def\mdots@{\mathinner.\nonscript\!.%
 \ifx\next,.\else\ifx\next;.\else\ifx\next..\else
 \nonscript\!\mathinner.\fi\fi\fi}
\let\ldots\mdots@
\let\cdots\mdots@
\let\dotso\mdots@
\let\dotsb\mdots@
\let\dotsm\mdots@
\let\dotsc\mdots@
\def\vdots{\vbox{\baselineskip2.8\p@ \lineskiplimit\z@
    \kern6\p@\hbox{.}\hbox{.}\hbox{.}\kern3\p@}}
\def\ddots{\mathinner{\mkern1mu\raise8.6\p@\vbox{\kern7\p@\hbox{.}}%
    \raise5.8\p@\hbox{.}\raise3\p@\hbox{.}\mkern1mu}}
\makeatother
\makeatletter
\let\Enumerate=\enumerate
\renewcommand{\enumerate}{\Enumerate%
\setlength{\@topsep}{0pt}
\setlength{\itemsep}{0pt}%
\setlength{\parskip}{0pt plus 1pt}%
\renewcommand{\theenumi}{\textup{(\alph{enumi})}}%
\renewcommand{\labelenumi}{\theenumi}%
}
\let\endEnumerate=\endenumerate
\renewcommand{\endenumerate}{\endEnumerate\unskip}
\makeatother
\makeatletter
\def\@seccntformat#1{\csname the#1\endcsname.\quad}
\makeatother
\newcommand{\authortitle}[2]{\author{#1}\title{#2}\markboth{#1}{#2}}
\newcommand{\auth}[2]{{#1, #2.}}

\newcommand{\art}[6]{{\sc #1, \rm #2, \it #3 \bf #4 \rm (#5), \mbox{#6}.}}

\newcommand{\book}[3]{{\sc #1, \it #2, \rm #3.}}
\newcommand{\AND}{{\rm and }}

\RequirePackage{amsthm}
\newtheoremstyle{descriptive}%
  {\topsep}   
  {\topsep}   
  {\rmfamily} 
  {}          
  {\bfseries} 
  {.}         
  { }         
  {}          
\newtheoremstyle{propositional}%
  {\topsep}   
  {\topsep}   
  {\itshape}  
  {}          
  {\bfseries} 
  {.}         
  { }         
  {}          
\newtheoremstyle{remarkstyle}%
  {\topsep}   
  {\topsep}   
  {\rmfamily}  
  {}          
  {\itshape} 
  {.}         
  { }         
  {}          
\theoremstyle{propositional}
\newtheorem{thm}{Theorem}[section]
\newtheorem{prop}[thm]{Proposition}
\newtheorem{lem}[thm]{Lemma}

\theoremstyle{descriptive}
\newtheorem{deff}[thm]{Definition}
\newtheorem{example}[thm]{Example}
\newtheorem{remark}[thm]{Remark}
\makeatletter
\renewenvironment{proof}[1][\proofname]{\par
  \pushQED{\qed}%
  \normalfont
  \trivlist
  \item[\hskip\labelsep
        \itshape
    #1\@addpunct{.}]\ignorespaces
}{%
  \popQED\endtrivlist\@endpefalse
}
\makeatother
\def\vint{\mathop{\mathchoice%
          {\setbox0\hbox{$\displaystyle\intop$}\kern 0.22\wd0%
           \vcenter{\hrule width 0.6\wd0}\kern -0.82\wd0}%
          {\setbox0\hbox{$\textstyle\intop$}\kern 0.2\wd0%
           \vcenter{\hrule width 0.6\wd0}\kern -0.8\wd0}%
          {\setbox0\hbox{$\scriptstyle\intop$}\kern 0.2\wd0%
           \vcenter{\hrule width 0.6\wd0}\kern -0.8\wd0}%
          {\setbox0\hbox{$\scriptscriptstyle\intop$}\kern 0.2\wd0%
           \vcenter{\hrule width 0.6\wd0}\kern -0.8\wd0}}%
          \mathopen{}\int}
{\catcode`p =12 \catcode`t =12 \gdef\eeaa#1pt{#1}}      
\def\accentadjtext#1{\setbox0\hbox{$#1$}\kern   
                \expandafter\eeaa\the\fontdimen1\textfont1 \ht0 }
\def\accentadjscript#1{\setbox0\hbox{$#1$}\kern 
                \expandafter\eeaa\the\fontdimen1\scriptfont1 \ht0 }
\def\accentadjscriptscript#1{\setbox0\hbox{$#1$}\kern   
                \expandafter\eeaa\the\fontdimen1\scriptscriptfont1 \ht0 }
\def\accentadjtextback#1{\setbox0\hbox{$#1$}\kern       
                -\expandafter\eeaa\the\fontdimen1\textfont1 \ht0 }
\def\accentadjscriptback#1{\setbox0\hbox{$#1$}\kern     
                -\expandafter\eeaa\the\fontdimen1\scriptfont1 \ht0 }
\def\accentadjscriptscriptback#1{\setbox0\hbox{$#1$}\kern 
                -\expandafter\eeaa\the\fontdimen1\scriptscriptfont1 \ht0 }

\DeclareMathOperator{\diam}{diam}
\DeclareMathOperator{\dist}{dist}

\newcommand{\simge}{\gtrsim}
\newcommand{\simle}{\lesssim}

\newcommand{\eps}{\varepsilon}

\newcommand{\s}{\sigma}

\newcommand{\cqd}{C_{\rm{qd}}}
\renewcommand{\phi}{\varphi}

\numberwithin{equation}{section}

\newenvironment{ack}{\medskip{\it Acknowledgement. }}{}

\begin{document}

\authortitle{Korte, Rogovin, Shanmugalingam, Takala}
           {Sharp conditions for sphericalization}
           \title{\Large Sharp conditions for preserving uniformity, doubling measure 
           and Poincar\'e inequality under sphericalization}
           \author{Riikka Korte, Sari Rogovin, Nageswari Shanmugalingam, Timo Takala}


\date{Preliminary version, \today}

\maketitle

\begin{abstract}
We study sphericalization, which is a mapping that conformally deforms the metric and the measure of an unbounded metric measure space so that the deformed space is bounded. The goal of this paper is to study sharp conditions on the deforming density function under which the sphericalization preserves uniformity of the space, the doubling property of the measure and the support of a Poincar\'e inequality. We also provide examples that demonstrate the sharpness of our conditions.
\end{abstract}

\ack{{\small R.K. is partially supported by the Research Council of Finland through project 360184. N.S. is partially supported by the National Science Foundation (US) grant DMS\#2348748. T.T. was supported by the Magnus Ehrnrooth foundation.}}

\medskip



\tableofcontents

\section{Introduction}

The stereographic projection and its inverse are very useful mappings in many situations. One of the reasons for the usefulness is that in addition to being conformal mappings and thus preserving many geometric properties of the space, the inverse of the stereographic projection transforms the unbounded plane into a bounded sphere.
During the past decades, similar mappings have been studied on more general metric measure spaces. Motivated by the stereographic projection, these kind of mappings that transform unbounded spaces into bounded ones are now called \emph{sphericalizations}, and the mappings into the other direction are called \emph{flattenings}. 
First formulated in the non-smooth setting in~\cite{BonkKleiner} and extracted into a more general structured framework in~\cite{BaloghBuckley}, a direct analog of sphericalization and flattening transformation has been explored from the point of view of uniform domains~\cite{BHX,HSX, LPZ} and potential theory~\cite{BBLi, DL15, DL17, LiShan}. Another type of transformation, this time following the framework of~\cite{BaloghBuckley}, was explored in~\cite{GibaraKS,GK} with the goal of developing tools useful in the study of fractional calculus on unbounded spaces. We point out here that our constructions and that of Balogh and Buckley \cite{BaloghBuckley} require that the space is path-connected, however, an alternate sphericalization is possible without requiring curves, as demonstrated in \cite{BBKRT}.

Given such utility of the construction considered in~\cite{BaloghBuckley}, it is natural to ask whether the transformations considered there guarantee preservation of certain properties such as the uniform domain property, and what conditions are needed in addition to obtain preservation of doubling property (geometric measure theoretic notion) and support of a Poincar\'e inequality (an analytic notion). It is also of interest to know how essential these conditions are in order to obtain such a preservation. This is the goal of the present paper.

With suitably chosen parameters, generalizations of transformations proposed in~\cite{BaloghBuckley} can preserve the Dirichlet $p$-energy and $p$-harmonic functions, and therefore these transformations are very useful tools in studying Dirichlet and Neumann boundary value problems in unbounded domains in metric measure spaces. The reason is that in order to be able to apply direct methods of calculus of variations, we need to be able to bound the $L^p$-norm of a Sobolev function by the Sobolev energy of the function. On bounded uniform domains this is always possible, if the space has a doubling measure and supports a Poincar\'e inequality. 
Thus by sphericalizing the space, we can transform a boundary value problem from an unbounded domain into a bounded domain where the direct methods of calculus of variations are easily applicable.
This kind of approach has been used previously in \cite{BBLi} and \cite{GibaraKS}. In \cite{BBLi} the authors construct a transformation of metric and measure that preserves $p$-minimizers and $p$-energy in Ahlfors regular metric spaces, while \cite{GibaraKS} studies the Dirichlet problem in doubling metric measure spaces by constructing a transformation that preserves doubling measures, Poincar\'e inequality, uniform domains and $p$-energy.

In this paper, we define the transformation of the metric and measure as follows. We fix an integrable Borel measurable function $\rho:(0,\infty)\to(0,\infty)$ and define a new metric $d_\rho$ by $d_\rho(x,y):=\inf_\gamma\int_\gamma \rho(d(\cdot, b))ds$, where the infimum is over all rectifiable curves with end points $x,y\in X$ and $b\in\partial X$ is a fixed base point. The transformed measure $\mu_{\rho}$ is defined as a weighted measure with weight $\rho(d(\cdot, b))^{\s}$, where $\s$ is a fixed positive parameter.

As studying different problems often requires different choices of the transformation of metric and measure, the goal of this project is to provide as general as possible conditions for $\rho$ in order to guarantee that the transformation is good in the sense that it preserves uniform domains, doubling property of the measure and $p$-Poincar\'e inequality.
Let us consider three conditions on $\rho$:

\begin{enumerate}
\renewcommand{\labelenumi}{\textbf{\theenumi}}
\renewcommand{\theenumi}{(\Alph{enumi})}
\item\label{condA} \label{rhodoubling}
There exists a constant $C_A$ such that whenever $0 < r \leq 2s+1$ and $0 < s \leq 2r+1$,
\begin{equation*} 
\rho(r) \leq C_A \rho(s).
\end{equation*}
\item\label{condB} \label{upperbound}
There exists a positive constant $C_B$ such that for every $r > 0$
\begin{equation*} 
\int_r^{\infty} \rho(t) dt
\leq C_B (r+1) \rho(r).
\end{equation*}
\item\label{condC} \label{condition-C}
There exists a positive constant $C_C$ such that for every $r > 0$
\begin{equation*} 
\int_{X\setminus B(b,r)} \rho(|x|)^\s d\mu(x)
\le C_C \rho(r)^\s \mu(B(b,r+1)).
\end{equation*}
Here the positive parameter $\s$ is the same as in the definition of the new measure.
\end{enumerate}
Conditions \ref{rhodoubling} and \ref{upperbound} were used by Balogh and Buckley in~\cite{BaloghBuckley}. The condition~\ref{rhodoubling} guarantees that we have some local control for the oscillation of $\rho$ and that $\rho$ cannot decay faster than at polynomial rate.
The condition~\ref{upperbound} enables control of distance to the new point $\infty$. Moreover, when equipped with the measure $\mu$, condition~\ref{condC} gives control of the transformed measure of balls centered at~$\infty$. 

Our main result tells that if $\rho$ is lower semicontinuous and satisfies the conditions \ref{rhodoubling}--\ref{condition-C}, then the transformation preserves, quantitatively, the uniformity of the space, the doubling property of the measure, and the $p$-Poincar\'e inequality.
\begin{thm}
\label{maintheorem}
Let $(X,d)$ be an unbounded uniform space and assume that $\rho$ satisfies conditions~\ref{condA} and~\ref{condB}. Then the following hold true.
\begin{enumerate}

\item 
The space $(X,d_{\rho})$ is bounded, the completion of $X$ with respect to $d_\rho$  adds to the completion with respect to $d$  exactly one point (called $\infty$), and $(X,d_{\rho})$ is a uniform space.

\item Let $(X,d)$ be equipped with a doubling measure $\mu$. Suppose that $\rho$ also satisfies Condition~\ref{condC}. Then $\mu_{\rho}$ is doubling in the space $(X,d_{\rho})$.

\item Let $(X,d)$ be equipped with a doubling measure $\mu$ such that $(X,d,\mu)$ supports a $p$-Poincaré inequality. Suppose that $\rho$ is also lower semicontinuous and satisfies Condition~\ref{condC}. 
Then $(X,d_{\rho},\mu_{\rho})$ supports a $p$-Poincaré inequality.

\end{enumerate}
\end{thm}

Our examples in Subsection~\ref{subsection:AB} demonstrate within the class of bounded, positive, integrable, and quasidecreasing functions $\rho$, the sharpness of the conditions to preserve uniformity: without Conditions~\ref{rhodoubling} and/or \ref{upperbound}, the uniformity property is not preserved in the case of a Euclidean half--plane $\mathbb R^2_+$. 
While these conditions need not be necessary in the setting of specific metric spaces, the example of open half-spaces demonstrates the need for these conditions in the absence of more specific structures.

The paper is organized as follows. In Section~\ref{sect-prelim} we present some preliminary results and introduce the sphericalized metric $d_\rho$. Then in Section~\ref{sect-uniformity} we focus on the uniformity property;
in Subsection \ref{subs:basic} we prove that if $\rho$ is bounded and quasidecreasing, then $(X,d_\rho)$ is bounded and the completion of $X$ with respect to $d_{\rho}$ adds exactly one point at infinity. In Subsection~\ref{subsect-uniformity} we show that if $\rho$ satisfies conditions~\ref{condA} and~\ref{condB}, then sphericalization preserves uniformity, and in Subsection~\ref{subsection:AB}, we provide examples that demonstrate the sharpness of conditions~\ref{condA} and~\ref{condB}. 

Next in Section~\ref{sect-doubling}, we focus on transforming doubling measures using an analogous sphericalization procedure, see Definition \ref{murhodef}.
In Subsection~\ref{sect-necessityC} we show that if $(X,d)$ is uniform and $\rho$ satisfies Conditions~\ref{rhodoubling} and~\ref{upperbound}, then to preserve the doubling property for the transformed measure $\mu_\rho$, $\rho$ must also satisfy Condition~\ref{condition-C}.
Then in Subsection \ref{subsect-doubling} we show that the doubling property is indeed preserved under sphericalization of uniform spaces, if $\rho$ satisfies conditions~\ref{condA}--\ref{condC}. 
As Condition~\ref{condition-C} requires the measure $\mu$, it is only needed in results that involve the new measure $\mu_{\rho}$. 
Therefore Condition~\ref{condition-C} is not required before Section~\ref{sect-doubling} and we postpone discussion of this condition to that section.
Conditions~\ref{rhodoubling} and~\ref{upperbound} are assumed to hold starting from Subsection~\ref{subsect-uniformity}. 

Finally, in Section~\ref{sect-poincare}, we prove that when $(X,d)$ is a uniform metric space with a doubling measure $\mu$ and supporting a $p$-Poincar\'e inequality for some $1\le p<\infty$, then the sphericalized space $(X,d_\rho,\mu_\rho)$ also supports a $p$-Poincar\'e inequality, if $\rho$ is lower semicontinuous and satisfies conditions~\ref{condA}--\ref{condC}. To prove this, we rely on the above established facts that $(X,d_\rho)$ is uniform and $\mu_\rho$ is doubling in $(X,d_\rho)$.

\section{Preliminaries} \label{sect-prelim}

\subsection{Standing assumptions}

The object of study in this paper is an unbounded incomplete metric space $(X,d)$ that is locally compact and rectifiably path-connected.
While we wish to focus on such metric spaces that are uniform domains with unbounded boundaries, much of the initial construction of the sphericalization does not need this assumption. We will start out with minimal assumptions on the unbounded metric space.
Starting from Section \ref{sect-doubling} we assume that the space is uniform and equipped with a doubling measure $\mu$.
In the last section of this paper, Section~\ref{sect-poincare}, we also assume that the metric measure space $(X,d,\mu)$ supports a $p$-Poincaré inequality.
The definitions of these properties are stated in the beginnings of Sections \ref{sect-uniformity}, \ref{sect-doubling} and \ref{sect-poincare}. 
By $\partial X$ we denote the set $\overline X^d\setminus X$, where $\overline X^d$ is the metric completion of $X$ with respect to the metric $d$.

We say that a function $f : (0,\infty) \rightarrow (0,\infty)$ is \emph{quasidecreasing}, if there is a constant $\cqd \geq 1$ such that 
\begin{equation}\label{eq:QDineq}
f(t) \leq \cqd f(s) \text{ whenever }0 < s \leq t.
\end{equation}

We let $\rho: (0,\infty) \rightarrow (0,\infty)$ be a function that is used as the metric density function in the definition of the new sphericalized metric $d_{\rho}$.
We will, for the rest of this paper, assume that $\rho$ is an integrable Borel measurable function.
Throughout the paper we let
\begin{equation*}
B(x,r) := \{ y \in X : d(x,y) < r \},
\end{equation*}
for $x \in \overline X^d$ and $r > 0$.

\subsection{Properties related to the metric density function \texorpdfstring{$\rho$}{}}

In this subsection we gather together some of the basic properties of the metric density function $\rho$ that are consequences of the conditions listed in the introduction. These properties are used throughout this paper.

Condition~\ref{rhodoubling} implies a \emph{doubling condition} for $\rho$; for $r>0$, we have that $C_A^{-1}\rho(r)\le \rho(2r)\le C_A\, \rho(r)$. Thus, whenever $k$ is a positive integer, we have
\begin{equation}\label{eq:power-doubling}
	C_A^{-k}\, \rho(r)\le \rho(2^k\, r)\le C_A^k\, \rho(r).
\end{equation}
The constant $C_A$ in Condition~\ref{rhodoubling} must be larger than $2$, as $\int_0^\infty\rho(t)\, dt$ is finite. 
Indeed, for $2<r\le 4=2^2$, we have from Condition~\ref{rhodoubling} that $\rho(r)\ge C_A^{-1}\, \rho(2)$. 
Proceeding inductively, we have that for each positive integer $j$, $\rho(r)\ge C_A^{-j}\, \rho(2)$, if $2^j<r\le 2^{j+1}$, and so
\[
\int_2^\infty\rho(t)\, dt=\sum_{j=1}^\infty\int_{2^j}^{2^{j+1}}\, \rho(t)\, dt\ge \sum_{j=1}^\infty 2^j\, C_A^{-j}\, \rho(2),
\]
from which we see that $\int_2^\infty\rho(t)\, dt$ is not finite, if $C_A\le 2$.

\begin{remark}
\label{remark-1}
In this remark we gather together consequences of Conditions~\ref{rhodoubling} and~\ref{upperbound} to the metric density function $\rho$.
\begin{enumerate}
\item[(i)]
If Condition \ref{rhodoubling} holds, then
\begin{equation*}
\int_r^{\infty} \rho(t) dt
\geq \int_r^{2r+1} \rho(t) dt
\geq (r+1) \inf_{r \leq s \leq 2r+1} \rho(s)
\geq \frac{1}{C_A} (r+1) \rho(r)
\end{equation*}
for every $r > 0$.
This further implies that $\rho$ is bounded and $\lim_{t \to \infty} t \rho(t) = 0$, as $\int_0^\infty\rho(t)\, dt$ is finite.
\item[(ii)]
If both \ref{rhodoubling} and \ref{upperbound} hold, then $(t+1) \rho(t)$ is quasidecreasing with associated constant $C_AC_B$. Therefore $\rho$ is also quasidecreasing with constant $\cqd=C_AC_B$. To see this, let $0 < r \leq s$. 
Then from the estimate above, we have
\begin{equation*}
(s+1) \rho(s)
\leq C_A \int_s^{\infty} \rho(t) dt
\leq C_A \int_r^{\infty} \rho(t) dt
\leq C_A C_B (r+1) \rho(r).
\end{equation*}
This also tells us that $C_A C_B \geq 1$.
\item[(iii)]
If $\rho$ is quasidecreasing, then
\begin{equation*}
\frac{1}{2} r \rho(r)
\leq \int_{\frac{1}{2} r}^r \cqd\, \rho(t) dt
\leq \cqd\, \int_{\frac{1}{2} r}^{\infty} \rho(t) dt.
\end{equation*}
This implies that $\lim_{t \to \infty} t \rho(t) = 0$, as $\int_0^\infty\rho(t)\, dt$ is finite.
\item[(iv)]
If $\rho$ is quasidecreasing and Condition~\ref{upperbound} holds, then  $(t+1) \rho(t)$ is quasidecreasing. Indeed let $0 < r \leq s$. If $s \leq 2r+1$, then $(s+1) \rho(s) \leq 2\cqd\, (r+1) \rho(r)$. 
If $s \geq 2r+1$, then $(s+1)/(s-r) \le 2$, and so
\begin{equation*}
(s+1) \rho(s)
= \frac{s+1}{s-r} \int_r^s \rho(s) dt
\leq 2 \int_r^s \cqd\, \rho(t) dt
\leq 2 \cqd\, C_B (r+1) \rho(r).
\end{equation*}
\end{enumerate}

\end{remark}
\vspace{0.5cm}
The following lemma is similar to \cite[Lemma 2.2]{BaloghBuckley}.
The difference is that in~\cite{BaloghBuckley}, the sphericalizing function $\rho$ is assumed to be continuous and the condition there that corresponds to our Condition~\ref{upperbound} is slightly different.
Therefore we couldn't directly apply their method.

\begin{lem}\label{BaloghLem}
Assume that Conditions \ref{rhodoubling} and \ref{upperbound} hold.
Then for $0 < r\le s$ we have
\[
(s+1)^{\eps+1}\rho(s)
\le C_A C_B (r+1)^{\eps+1} \rho(r),
\]
where $\eps=\frac{1}{C_A C_B}$.
\end{lem}

\begin{proof}
From Remark \ref{remark-1}(ii) we know that $C_AC_B\ge 1$, and so with our choice of $\eps$ we get $0 < \eps \leq 1$. Then from the concavity of the function $(0,\infty)\ni x\to x^\eps$ we see that 
\begin{equation*}
(r+1+\delta)^{\eps} - (r+1)^{\eps}
\leq \delta \eps (r+1)^{\eps-1}
\end{equation*}
for every $\delta \geq 0$.
Especially if $0 \leq \delta \leq r+1$, we get
\begin{align*}
&(r+1)^{\eps} \int_r^{\infty} \rho(t) dt - (r+\delta+1)^{\eps} \int_{r+\delta}^{\infty} \rho(t) dt
\\
= &((r+1)^{\eps} - (r+\delta+1)^{\eps}) \int_r^{\infty} \rho(t) dt + (r+\delta+1)^{\eps} \int_r^{r+\delta} \rho(t) dt
\\
\geq &- \delta \eps (r+1)^{\eps-1} C_B (r+1) \rho(r) + (r+1)^{\eps} \delta \frac{1}{C_A} \rho(r)
= 0,
\end{align*}
where we used \ref{rhodoubling}, \ref{upperbound} and our choice of $\eps$.
Thus the function $r \rightarrow (r+1)^{\eps} \int_r^{\infty} \rho(t) dt$ is decreasing and therefore by Remark~\ref{remark-1}(i) and~\ref{upperbound}, when $0<r\le s$,
\begin{align*}
(s+1)^{\eps+1} \rho(s)
&\leq (s+1)^{\eps} C_A \int_s^{\infty} \rho(t) dt
\leq (r+1)^{\eps} C_A \int_r^{\infty} \rho(t) dt
\\
&\leq C_A C_B (r+1)^{\eps+1} \rho(r).
\qedhere
\end{align*}
\end{proof}

The next lemma shows that if $\rho$ satisfies the condition $\int_r^\infty\rho(t)\, dt \simeq (r+~1)\rho(r)$, then $\rho$ must satisfy Condition~\ref{rhodoubling}. 
As a consequence of the lemma and Remark~\ref{remark-1}(i) we know that this condition on $\int_r^\infty\rho(t)\, dt$ is equivalent to Conditions~\ref{condA} and~\ref{condB} together. A benefit of this condition on $\int_r^\infty\rho(t)\, dt$ is the ease of obtaining estimates for the metric $d_{\rho}$ between points in $X$ and the ``point at $\infty$'' that is the new boundary point of $X$ under the transformed metric, as in Lemma~\ref{lem-4}.

\begin{lem}
\label{lem-lowerbound}
Assume that there exists a constant $C \geq 1$ such that
\begin{equation}
\label{equivalentcondition}
\frac{1}{C} (r+1) \rho(r)
\leq \int_r^{\infty} \rho(t) dt
\leq C (r+1) \rho(r)
\end{equation}
for every $r > 0$.
Then there exists a constant $C' \geq 1$, depending only on $C$, such that $\rho(r) \leq C' \rho(s)$ whenever $r , s > 0$ and $s \leq 2r+1$.
\end{lem}

\begin{proof}
For every $0 < s \leq r$ we clearly have $(r+1) \rho(r) \leq C^2 (s+1) \rho(s)$ from our hypothesis, and therefore $\rho(r) \leq C^2 \rho(s)$, meaning that $\rho$ is quasidecreasing with constant $\cqd=C^2$. This also proves the claim for the case $0<s\le r$. Now assume that $0 < r < s \leq 2r + 1$.
For every $r > 0$ we have
\begin{align*}
\frac{1}{C} (r+1) \rho(r)
&\leq \int_r^{\infty} \rho(t) dt
\\
&= \int_r^{r + \frac{1}{2 C^3}(r+1)} \rho(t) dt + \int_{r + \frac{1}{2 C^3}(r+1)}^{\infty} \rho(t) dt
\\
&\leq \frac{1}{2 C^3} (r+1) C^2 \rho(r) + C \left( 1 + \frac{1}{2 C^3} \right) (r+1) \rho \left( r + \frac{1}{2 C^3}(r+1) \right),
\end{align*}
where we have used the above-established fact that $\rho$ is quasidecreasing, and therefore
\begin{equation*}
\rho(r)
\leq \left( 2 C^2 + \frac{1}{C} \right) \rho \left( r + \frac{1}{2 C^3}(r+1) \right).
\end{equation*}
By iterating the previous result we get for each positive integer $m$ that
\[
\rho(r)\le \left(2C^2+\frac{1}{C}\right)^m\, \rho\left(\left(1+\frac{1}{2C^3}\right)^m\, (r+1) - 1\right).
\]
Then by choosing $m$ so that
\[
\left(1+\frac{1}{2C^3}\right)^{m-1}\, (r+1)-1
< 2r+1
\le \left(1+\frac{1}{2C^3}\right)^m\, (r+1)-1,
\]
we obtain
\begin{align*}
\rho(r)
&\leq \left(2C^2+\frac{1}{C}\right)^m\, \rho\left(\left(1+\frac{1}{2C^3}\right)^m\, (r+1) - 1\right)
\\
&\leq \left( 2 C^2 + \frac{1}{C} \right)^{1 + \frac{1}{\log_2 \left( 1+\frac{1}{2 C^3} \right) }} C^2 \rho(s).
\qedhere
\end{align*}

\end{proof}

\begin{remark}

Remark \ref{remark-1}(i) and Lemma \ref{lem-lowerbound} tell us that instead of Conditions \ref{rhodoubling} and \ref{upperbound} we could equivalently assume that \eqref{equivalentcondition} holds for every $r > 0$.

\end{remark}

\begin{example}
Consider the example function $\rho(t) = (t+2)^{\alpha} (\log(t+2))^{\beta}$ with parameters $\alpha$ and $\beta$.
If $\alpha > -1$, then $\rho$ is not integrable.
If $\alpha < -1$, then $\rho$ is integrable and it satisfies \ref{rhodoubling} and \ref{upperbound}.
Now let $\alpha = -1$.
If $\beta \geq -1$, then $\rho$ is not integrable.
If $\beta < -1$, then $\rho$ is integrable, bounded and (quasi)decreasing and it satisfies \ref{rhodoubling}, but it does not satisfy \ref{upperbound}.
Conversely the function $\rho(t) = e^{-t}$ is integrable, bounded and (quasi)decreasing and it satisfies \ref{upperbound}, but it does not satisfy \ref{rhodoubling}.
\end{example}
See Section \ref{sect-doubling} for examples demonstrating the role of Condition \ref{condition-C}.

\subsection{The sphericalized metric \texorpdfstring{$d_{\rho}$}{}}

We are now ready to define the \emph{sphericalized metric} $d_\rho$ associated with the function~$\rho$. To do so, we fix a base point $b \in \partial X$ and set $|x| := d(b,x)$.
Our preliminary standing assumption is that $\rho:(0,\infty)\to(0,\infty)$ is an integrable Borel measurable function, but as we progress in the paper, we will add the assumptions \ref{rhodoubling}, \ref{upperbound} and~\ref{condition-C}.

A \emph{curve} is a continuous mapping $\gamma\colon [s,t] \rightarrow X$. 
We call the image set $\gamma([s,t])$ also a curve and denote it just by $\gamma$.
The length of $\gamma$, with respect to the metric $d$, is defined by
\begin{equation*}
\ell_d(\gamma)
:= \sup \sum_{j=1}^n d(\gamma(t_{j-1}),\gamma(t_j)),
\end{equation*}
where the supremum is taken over all partitions $s = t_0 < t_1 < ... < t_n = t$ of the interval $[s,t]$. A curve is said to be \emph{rectifiable} with respect to the metric $d$, if $\ell_d(\gamma)<\infty$. For a rectifiable curve $\gamma\colon [s,t] \to X$ there is the associated \emph{length function} $s_d\colon[s,t]\to[0,\ell_d(\gamma)]$ defined by $s_d(u):=\ell_d(\gamma\vert_{[s,u]})$, where $\gamma\vert_{[s,u]}$ is the restriction of $\gamma$ on the interval $[s,u]$. This $s_d$ is absolutely continuous if and only if $\gamma$ is absolutely continuous. 
The parametrization of a rectifiable curve $\gamma$ by arc-length with respect to the metric $d$ is the reparametrization of $\gamma$ as $\gamma_d : [0,\ell_d(\gamma)] \rightarrow X$ so that for every $0 \leq u \leq \ell_d(\gamma)$, $\ell_d(\gamma_d\vert_{[0,u]}) = u$, and $\gamma (u)=\gamma_d\circ s_d(u)$ for all $u\in [s,t]$. 
Note that this kind of reparametrization exists for all curves that are rectifiable with respect to the metric $d$.

The \textit{path integral} of a non-negative Borel function $g\colon X\to[0,\infty]$ over rectifiable curves is defined by using the param\-etrization by arc-length with respect to the metric $d$, i.e.
\begin{equation*}
\int_{\gamma} g\, ds
:= \int_0^{\ell_d(\gamma)} g(\gamma_d(u)) du.
\end{equation*}
In case $\gamma\colon [s,t]\to X$ is absolutely continuous, also $s_d\colon[s,t]\to [0,\ell_d(\gamma)]$ is absolutely continuous and therefore the derivative $s_d'$ exists almost everywhere. Thus by the change of variables
\begin{equation}\label{eqn:pathintegral}
\begin{split}
    \int_\gamma g\,ds&=\int_0^{\ell_d(\gamma)} g(\gamma_d(u)) du \\
    &=\int_s^t g(\gamma_d(s_d(u)))s_d'(u) du=\int_s^t g(\gamma(u))s_d'(u) du.
    \end{split}
\end{equation}
We refer the interested reader to~\cite[Theorem~4.1]{Vai} for more on path integrals in the Euclidean setting, and to~\cite[Chapter~5]{HKSTbook} for the more general metric setting.

\begin{deff}
We define the new sphericalized metric $d_\rho$ by
\begin{equation*}
d_{\rho}(x,y)
:= \inf_{\gamma \in \Gamma(x,y)} \int_{\gamma} \rho(|\cdot|) ds.
\end{equation*}
Here $\Gamma(x,y)$ is the set of rectifiable curves with end points $x$ and $y$.
\end{deff}

The next lemma describes conditions under which $d_\rho$ is a metric on $X$.

\begin{lem}
\label{lem-6}$\,$
\begin{enumerate}
\item[(i)]
Suppose that for every  $\eps>0$ and for every $M \geq \eps$ we have $\inf_{\eps \leq r \leq M} \rho(r) > 0$.
Then $d_{\rho}$ is symmetric, satisfies the triangle inequality and $d_{\rho}(x,y) = 0$ if and only if $x=y$.
\item[(ii)]
Suppose that for every $\eps>0$ we have $\sup_{r \geq \eps} \rho(r) < \infty$.
Then $d_{\rho}(x,y) < \infty$ for every $x , y \in X$.
\end{enumerate}
\noindent In particular, if for every $\eps>0$ and for every $M \geq \eps$ we have $\inf_{\eps \leq r \leq M} \rho(r) > 0$ and $\sup_{r \geq \eps} \rho(r) < \infty$, then $(X,d_{\rho})$ is a metric space.
\end{lem}

Note that the condition in Lemma~\ref{lem-6}(i) will hold, if we assume that $\rho$ is quasidecreasing or that $\rho$ is lower semicontinuous.
Also the condition in Lemma~\ref{lem-6}(ii) will hold, if we assume that $\rho$ is quasidecreasing.
When $\rho$ satisfies both \ref{rhodoubling} and \ref{upperbound}, by Remark \ref{remark-1}(ii) we know that $\rho$ is quasidecreasing.

\begin{proof}
Proof of (i):
From the definition it is clear that $d_{\rho}$ is symmetric, it satisfies the triangle inequality and $d_{\rho}(x,x)=0$ for every $x \in X$. Assume now that $x \neq y$. Let $\gamma$ be a curve that connects $x$ to $y$.
Let $r_0 := \min\{|x|/2,d(x,y)\}$. Then
\begin{equation*}
\int_{\gamma} \rho(|\cdot|) ds
\geq \int_{\gamma \cap B(x,r_0)} \rho(|\cdot|) ds
\geq r_0 \inf_{z \in \gamma \cap B(x,r_0)} \rho(|z|)
\geq r_0 \inf_{\frac{1}{2}|x| \leq r \leq \frac{3}{2}|x|} \rho(r).
\end{equation*}
By taking infimum over all curves we get $d_{\rho}(x,y) > 0$.

Proof of (ii):
Since $(X,d)$ is rectifiably path-connected, there exists a rectifiable curve $\gamma \subset X$ connecting $x$ to $y$. 
Then because $[0,\ell_d(\gamma)]$ is compact and the function $t\mapsto|\gamma(t)|$ is continuous, this function attains its minimum, which is positive because $b \notin X$. Let $\delta$ be this minimum. Then
\begin{equation*}
d_{\rho}(x,y)
\leq \int_{\gamma} \rho(|\cdot|) ds
\leq \ell_d(\gamma) \sup_{r \geq \delta} \rho(r)
< \infty.
\qedhere
\end{equation*}
\end{proof}

One of the three foci of this paper is the notion of uniform domains, which requires that we deal with a locally compact but non-complete metric space, and then we consider the boundary of the space to be the collection of all new points obtained in the completion of the space. Distance from points in the space $X$ to the boundary plays a key role in the notion of uniform domains.
We denote the distance to the boundary with respect to the metric $d$ by 
\begin{equation} \label{eqn:bdry-dist}
d_X(z):=\dist_d(z,\partial X):=\inf\{d(z,w)\, :\, w\in\partial X\},
\end{equation}
and the distance to the boundary with respect to the transformed metric $d_\rho$ by 
\[
d_{X,\rho}(z) := \dist_{\rho}(z,\partial_{\rho} X):=\inf\{d_\rho(z,w)\, :\, w\in\partial_\rho X\}.
\]
Here $\partial_{\rho} X := \overline X^{d_{\rho}} \setminus X$ with $\overline{X}^{d_\rho}$ the metric completion of $X$ with respect to the metric $d_\rho$. Recall that $\partial X = \overline X^d \setminus X$ with $\overline{X}^d$ the metric completion of $X$ with respect to the metric $d$.
Section~\ref{sect-uniformity} gives the definition of uniform domains.

We use the notation $\ell_\rho(\gamma)$ for the length of a curve $\gamma:[s,t]\to X$ under the metric $d_\rho$ i.e.,
\begin{equation*}
\ell_{\rho}(\gamma)
:= \ell_{d_{\rho}}(\gamma)
:= \sup \sum_{j=1}^n d_{\rho}(\gamma(t_{j-1}),\gamma(t_j)),
\end{equation*}
where the supremum is over all finite partitions $s=t_0<t_1<\cdots<t_n=t$ of the interval $[s,t]$.
This implies that
\begin{equation}
\label{rholength}
\ell_{\rho}(\gamma)
\leq \int_{\gamma} \rho(|\cdot|) ds.
\end{equation}
If $\rho$ is continuous and $X$ has the same topology with the length metric and the metric $d$, then we have equality in~\eqref{rholength}, see~\cite[Proposition A.7]{BHK}.
However, except in Section~\ref{sect-poincare}, we don't assume continuity or even lower semicontinuity of $\rho$, so we only have the inequality.

We complete this section with a lower bound for the metric $d_{\rho}$. Note that at this point we are only assuming that $\rho$ is an integrable positive Borel measurable function.

\begin{lem}
\label{lem-2}
Let $x,y \in X$ such that $|x| \leq |y|$. Then
\begin{equation*}
d_{\rho}(x,y)
\geq \int_{|x|}^{|y|} \rho(t) dt.
\end{equation*}
\end{lem}

\begin{proof}
Let $\gamma \in \Gamma(x,y)$ and assume that $\gamma$ is parametrized by arc-length with respect to $d$.
Define $g\colon [0,\ell_d(\gamma)]\to [0,\infty)$ with $g(t)=|\gamma(t)|$. Then $g$ attains every value between $|x|$ and $|y|$, and thus
\begin{equation*}
    \{s\in [|x|,|y|]:\rho(s)>t\}\subset g(\{s\in [0,\ell_d(\gamma)]:\rho(|\gamma(s)|)>t\}). 
\end{equation*}
Moreover, since $g$ is 1-Lipschitz, we have $\mathcal{H}^1(g(A))\le \mathcal{H}^1(A)$ for $A\subset [0,\ell_d(\gamma)]$ (see for example \cite[Proposition 11.18]{Fol}). 
Here $\mathcal{H}^1$ is the $1$-dimensional Hausdorff measure. Thus together with Cavalieri's principle (i.e. a special case of Fubini's theorem \cite[Proposition 6.24]{Fol}) 
\begin{align*}
\int_{|x|}^{|y|}\rho(t)\,dt&=\int_0^\infty\mathcal{H}^1(\{s\in [|x|,|y|]:\rho(s)>t\})\,dt\\
&\le \int_0^\infty\mathcal{H}^1\left(g(\{s\in [0,\ell_d(\gamma)]:\rho(|\gamma(s)|)>t\})\right)\,dt\\
&\le \int_0^\infty\mathcal{H}^1\left(\{s\in [0,\ell_d(\gamma)]:\rho(|\gamma(s)|)>t\}\right)\,dt\\
&=\int_0^{\ell_d(\gamma)} \rho(|\gamma(t)|)\,dt.
\end{align*}
Since this holds for all curves $\gamma\in \Gamma(x,y)$, we get the result by taking the infimum over all $\gamma$.
\end{proof}

\section{Uniformity}
\label{sect-uniformity}

This section is devoted to the geometric notion of uniform domain, also called uniform space. We explore how this geometric property is transformed under the new metric $d_\rho$. Since the notion of uniformity is a purely metric notion, we do not consider any measure in this section. The main result of this section, Theorem~\ref{thm-uniformity}, forms part of the main theorem of this paper, Theorem \ref{maintheorem}.

\begin{deff}
Let $C_U \geq 1$. For a metric space $(X,d)$, a curve $\gamma \subset \overline X^d$ connecting points $x,y\in\overline X^d$ is said to be a \emph{$C_U$-uniform curve in $(X,d)$}, if
\begin{itemize}
\item[(i)] $\ell_d(\gamma) \leq C_U d(x,y)$ and
\item[(ii)] $\min\{ \ell_d(\gamma_{xz}) , \ell_d(\gamma_{zy}) \}
\leq C_U d_X(z)$ for every $z \in \gamma$.
\end{itemize}
A metric space $(X,d)$ is \emph{uniform}, if $(X,d)$ is locally compact but not complete and there exists a constant $C_U \geq 1$ such that for every two points $x , y \in X$ there exists a curve $\gamma \subset X$ connecting $x$ to $y$ such that $\gamma$ is $C_U$-uniform in $(X,d)$.
\end{deff}

We refer the reader to \eqref{eqn:bdry-dist} for the definition of $d_X$.
The first condition is known as \emph{quasiconvexity} and the second condition is known as the \emph{twisted cone condition}.
Here $\gamma_{xz}$ denotes any subarc of $\gamma$ between $x$ and $z$.
While a uniform curve $\gamma$ can intersect $\partial X$, by the twisted cone condition it can do so only at its terminal points.

If $(X,d)$ is uniform, then $(\overline{X}^d,d)$ is proper by \cite[Proposition 2.20]{BHK}.
The goal of this section is to show that the uniformity of $(X,d)$ implies the uniformity of  $(X,d_{\rho})$.
Then it also follows from Theorem \ref{thm-extension} below that any point $x \in X$ can be connected to the point $\infty$ by a curve that is uniform in $(X,d_{\rho})$.
The point $\infty$ is defined right before Theorem \ref{thm-9} below.

First, in Subsection~\ref{subsect-homeo}, we show that $(X,d)$ and $(X,d_\rho)$ are homeomorphic and the set of compact
rectifiable curves is the same in these two spaces.
Then, in Subsection~\ref{subsect-extension}, we prove a general extension result, a folklore that is well-known to experts, for uniform spaces. 
In Subsection \ref{subs:basic} we show that $\partial_\rho X=\partial X\cup\{\infty\}$ and $(X,d_{\rho})$ is bounded provided that $(X,d)$ is uniform.
In Subsection~\ref{subsection:AB}, we prove the sharpness of Conditions~\ref{rhodoubling} and \ref{upperbound} for $\rho$.
Finally, in Subsection \ref{subsect-uniformity}, we prove the main result of this section, which is that the uniformity of $(X,d)$ implies the uniformity of $(X,d_{\rho})$.

\subsection{Length of curves in quasiconvex spaces}
\label{subsect-homeo}

In this subsection we show that the change of metric creates a homeomorphism and that it preserves rectifiability of curves. For these results it is enough that the space is quasiconvex; uniformity is not needed.

\begin{prop}\label{homeo}
Suppose that $(X,d)$ is quasiconvex and that for every $\eps > 0$ and for every $M \geq \eps$ we have $\inf_{\eps \leq r \leq M} \rho(r) > 0$ and $\sup_{r \geq \eps} \rho(r) < \infty$. 
Then the identity map $\mathrm{id}\colon(X,d)\to(X,d_\rho)$ is a homeomorphism.
\end{prop}

\begin{proof}
By Lemma \ref{lem-6} the function $d_{\rho}$ is a metric in $X$.
The identity map is trivially a bijection, so what we need to show is the continuity of the map and its inverse.

To prove the continuity let $x \in X$ and $\tilde\eps > 0$. We set
\begin{equation*}
\delta := \min \left\{ \frac{|x|}{2C} , \frac{\tilde\eps}{C \sup_{r \geq |x|/2} \rho(r)} \right\},
\end{equation*}
where $C$ is the quasiconvexity constant. Now if $d(x,y) < \delta$ and $\gamma$ is a quasiconvex curve connecting $x$ to $y$, then $\ell_d(\gamma) \leq C d(x,y) < C \delta \leq |x|/2$. Thus $|z| \geq |x|/2$ for every $z \in \gamma$ and therefore
\begin{equation*}
d_\rho(x,y)
\le\int_\gamma\rho(|\cdot|)ds
\le\sup_{r \geq |x|/2} \rho(r) \ell_d(\gamma)
\le C\sup_{r \geq |x|/2} \rho(r) d(x,y)
< \tilde\eps.
\end{equation*}

To prove the continuity of the inverse, fix $x\in X$ and $\tilde\eps>0$. We set $m=\inf_{\frac{|x|}{2} \leq r \leq 3\frac{|x|}{2}}\rho(r)$ and $\delta=m \min \{ |x|/2 , \tilde\eps \}$. If $y\in X$ such that $d_\rho(x,y)<\delta$, then there exists $\gamma\in \Gamma(x,y)$ for which $\int_\gamma\rho(|\cdot|)ds <\delta$. The curve $\gamma$ lies completely inside $B(x,\frac{|x|}{2})$, because otherwise
\[
\delta>\int_\gamma\rho(|\cdot|)ds\ge\int_{\gamma\cap B(x, \frac{|x|}{2})}\rho(|\cdot|)ds\ge m\cdot \frac{|x|}{2} \geq \delta.
\]
Then we obtain
\[
d(x,y)\le \frac{1}{m}m\, \ell_d(\gamma)\le\frac{1}{m}\int_\gamma\rho(|\cdot|)ds<\frac{1}{m}\delta\leq\tilde\eps
\]
and therefore we have the continuity.
\end{proof}

\begin{lem}
\label{lem-10}
Suppose that for every $\eps>0$ and for every $M \geq \eps$ we have that $\inf_{\eps \leq r \leq M} \rho(r) > 0$ and $\sup_{r \geq \eps} \rho(r) < \infty$. 
Let $\gamma \subset X$ be a curve that is rectifiable with respect to the metric $d_{\rho}$.
Let $0 < \delta \leq r$ and assume that $\gamma$ connects points $x , y \in X$ such that $\delta \leq |z| \leq r$ for every $z \in \gamma$. 
Then for every $R > r$ and for every $0 < \delta' < \delta$ we have
\begin{equation*}
\ell_{\rho}(\gamma)
\geq d(x,y) \inf_{\delta' \leq r' \leq R} \rho(r').
\end{equation*}
\end{lem}

\begin{proof}
By Lemma~\ref{lem-6}, the function $d_{\rho}$ is a metric in $X$.
Let 
$\gamma(t)$ be the parametrization of $\gamma$ by arc-length with respect to the metric $d_{\rho}$. We divide $\gamma$ into a partition  
$0 = t_0 < t_1 < t_2 < ... < t_n = \ell_{\rho}(\gamma)$ such that 
\[
t_i - t_{i-1} \leq \min \left\{ \int_{\delta'}^{\delta} \rho(t) dt , \int_r^R \rho(t) dt \right\}
\] 
for every $1 \leq i \leq n$.
Let $0 < \eps <\min \left\{ \int_{\delta'}^{\delta} \rho(t) dt , \int_r^R \rho(t) dt \right\}$ and set $x_i := \gamma(t_i)$ for each $i=0,\cdots, n$. 
Then for every $i$ there exists a curve $\beta_{i,\eps}$ connecting the points $x_{i-1}$ and $x_i$ such that
\begin{align*}
\int_{\beta_{i,\eps}} \rho(|\cdot|)ds
\leq d_{\rho}(x_{i-1} , x_i) + \eps
&< t_i - t_{i-1} + \min \left\{ \int_{\delta'}^{\delta} \rho(t) dt , \int_r^R \rho(t) dt \right\}
\\
&\leq 2 \min \left\{ \int_{\delta'}^{\delta} \rho(t) dt , \int_r^R \rho(t) dt \right\}.
\end{align*}
Since $|x_i|,|x_{i-1}|\le r$ by the assumption $|z|\le r$ for every 
$z\in\gamma$, it follows that if 
$|z| \geq R$ for some $z \in \beta_{i,\eps}$,  
then by Lemma \ref{lem-2},
\begin{equation*}
\int_{\beta_{i,\eps}} \rho(|\cdot|)ds
\geq d_{\rho}(x_{i-1} , z) + d_{\rho}(z , x_i)
\geq \int_{|x_{i-1}|}^{|z|} \rho(t) dt + \int_{|x_i|}^{|z|} \rho(t) dt
\geq 2 \int_r^R \rho(t) dt,
\end{equation*}
which is a contradiction. Therefore $|z| < R$ for every $z \in \beta_{i,\eps}$.
Similarly if $|z| \leq \delta'$ for some $z \in \beta_{i,\eps}$, then by Lemma \ref{lem-2} and the assumption that $|z|\ge \delta$ for every $z\in\gamma$,
\begin{equation*}
\int_{\beta_{i,\eps}} \rho(|\cdot|)ds
\geq d_{\rho}(x_{i-1} , z) + d_{\rho}(z , x_i)
\geq \int_{|z|}^{|x_{i-1}|} \rho(t) dt + \int_{|z|}^{|x_i|} \rho(t) dt
\geq 2 \int_{\delta'}^{\delta} \rho(t) dt,
\end{equation*}
which is a contradiction. Therefore $|z| > \delta'$ for every $z \in \beta_{i,\eps}$.

Let $\beta_{\eps}$ be the union of the curves $\beta_{i,\eps}$. Then
\begin{align*}
\ell_{\rho}(\gamma)
&\geq \sum_{i=1}^n d_{\rho}(x_{i-1} , x_i) 
\geq \sum_{i=1}^n \left(\int_{\beta_{i,\eps}} \rho(|\cdot|)ds - \eps\right)
= \int_{\beta_{\eps}} \rho(|\cdot|)ds - n \eps
\\
&\geq \int_{\beta_{\eps}} \inf_{z \in \beta_{\eps}} \rho(|z|) ds - n \eps
\geq \ell_d(\beta_{\eps}) \inf_{\delta' < r' < R} \rho(r') - n \eps
\\
&\geq d(x,y) \inf_{\delta' < r' < R} \rho(r') - n \eps.
\end{align*}
By letting $\eps \to 0$ we get the result.
\end{proof}

Should $\rho$ be lower semicontinuous, we can take $R=r$ and $\delta' = \delta$ in the above lemma, see Lemma \ref{lem:lsc} below.

\begin{prop}
\label{rectifiability}
Suppose that $(X,d)$ is quasiconvex and that for every $\eps > 0$ and for every $M \geq \eps$ we have $\inf_{\eps \leq r \leq M} \rho(r) > 0$ and $\sup_{r \geq \eps} \rho(r) < \infty$. 
Then a curve $\gamma\colon[s,t]\to X$ is rectifiable with respect to the metric $d$ if and only if it is rectifiable with respect to the metric $d_\rho$.
\end{prop}

\begin{proof}
Notice that because $(X,d)$ and $(X,d_\rho)$ are homeomorphic by Proposition~\ref{homeo}, our map $\gamma$ is continuous in both spaces and thus a curve in both spaces.
Let $\delta := \min_{z \in \gamma} |z|$, which is positive, because $\gamma$ is a compact subset of $X$.
Let $M=\sup_{r \geq \delta} \rho(r) < \infty$.
Then, by \eqref{rholength},
\[
\ell_{\rho}(\gamma)
\leq \int_{\gamma} \rho(|\cdot|) ds
\leq M \ell_d(\gamma).
\]
Thus if $\gamma$ is rectifiable with respect to the metric $d$, then it is also rectifiable with respect to the metric $d_\rho$.

To see the other direction, suppose that $\gamma$ is rectifiable with respect to the metric $d_\rho$. Note that since $\gamma\colon[s,t]\to X$ is continuous also with respect to $d$, the trajectory of $\gamma$ is compact and so there is some $r>0$ such that $|z| \leq r$ for every $z \in \gamma$.
Define $m=\inf_{\delta / 2 \leq r' \le r+1} \rho(r')$. Let $s=t_0<t_1<\cdots<t_n=t$ be a partition of the interval $[s,t]$. Then by Lemma~\ref{lem-10}
\[
\sum_{j=1}^n d(\gamma(t_{j-1}),\gamma(t_j))\le\frac{1}{m}\sum_{j=1}^n \ell_\rho(\gamma|_{[t_{j-1},t_j]})=\frac{1}{m}\ell_\rho(\gamma)
\]
and by taking supremum over all partitions we obtain $\ell_d(\gamma)\le\frac{1}{m}\ell_\rho(\gamma)$ and thus $\gamma$ is rectifiable with respect to the metric $d$ also.
\end{proof}

\subsection{Existence of uniform curves to the boundary}
\label{subsect-extension}

The notion of uniformity guarantees the existence of uniform curves connecting pairs of points in the space. In this subsection we extend this existence to points on the boundary of the space as well. While this is a folklore that is well-known to the experts, we provide the detailed proof for the convenience of the reader. 

\begin{thm}
\label{thm-extension}
Let $(X,d)$ be a $C_U$-uniform space. 
Then any two points $x , y \in \overline X^d$ can be connected by a $C_U$-uniform curve in $(X,d)$.
\end{thm}

\begin{proof}
If $x , y \in X$, then the claim is trivial.
Suppose that $x \in X$ and $y\in \partial X$. Let us fix a sequence of points $y_i\in X$ such that $\lim_{i\to\infty}d(y,y_i)=0$.
We may assume without loss of generality that $d(y,y_i)\le 1$ for each $i$.

For each $y_i$ there exists a $C_U$-uniform curve $\gamma_i$ that connects $x$ to $y_i$.
We have that $\gamma_i : [0,\ell_d(\gamma_i)] \rightarrow X$ such that $\gamma_i(0) = x$ and $\gamma_i(\ell_d(\gamma_i)) = y_i$.
From the quasiconvexity condition of $\gamma_i$ we get $\ell_d(\gamma_i) \leq C_U d(x,y_i) \leq C_U (d(x,y) + d(y,y_i)) \leq C_U(d(x,y)+1) := L$. Therefore we choose to extend the domain of $\gamma_i$ so that $\gamma_i : [0,L] \rightarrow X$ such that $\gamma_i(t) = y_i$ for each $t$ with $\ell_d(\gamma_i) \leq t \leq L$.
Note that as $\gamma_i$ is parametrized by arc-length with respect to the metric $d$, this new curve is $1$-Lipschitz as a map from $[0,L]$. 
Also by the twisted cone condition on $\gamma_i$, we have
\begin{equation*}
\min\{t,\ell_d(\gamma_i)-t\}
\leq C_U \inf_{z \in \partial X} d(\gamma_i(t),z)
\end{equation*}
for every $0 \leq t \leq \ell_d(\gamma_i)$. Indeed this also holds, if $\ell_d(\gamma_i) < t \leq L$, because then the left-hand side of the above inequality is negative.

Thus we have a sequence of $1$-Lipschitz functions $\gamma_i : [0,L] \rightarrow X$, and so this sequence is equicontinuous. 
This sequence is bounded, because $d(x,\gamma_i(t)) \leq \ell_d({\gamma_i}_{[0,t]}) \leq \ell_d(\gamma_i) \leq L$ for every $i$ and for every $0 \leq t \leq L$.
By the Arzel\`{a}-Ascoli theorem, we conclude that there exists a subsequence $\gamma_{i_k}$ that converges uniformly to a function $\gamma:[0,L]\to\overline{X}^d$.
Here we also used the fact that $(\overline{X}^d,d)$ is proper.
We claim that $\gamma$ is a $C_U$-uniform curve in $(X,d)$ and connects $x$ to $y$.

As a uniform limit of a sequence of $1$-Lipschitz maps, $\gamma$ is necessarily $1$-Lipschitz. It is also clear that $\gamma(0)=x$. Note that 
\begin{equation*}
0 = \lim_{k \to \infty} d(\gamma_{i_k}(L),\gamma(L))
= \lim_{k \to \infty} d(y_{i_k},\gamma(L))
= d(y,\gamma(L)).
\end{equation*}
Thus $\gamma$ is a curve that connects $x$ to $y$.

To prove the quasiconvexity, let $0 = t_0 < t_1 < ... < t_n = L$ be a partition of the interval $[0,L]$. 
Because of the uniform convergence, for any $\eps > 0$ there exists $K \in \mathbb{N}$ such that $\sup_{t \in [0,L]} d(\gamma_{i_k}(t),\gamma(t)) < \eps$ whenever $k \geq K$. For such $k$ we have
\begin{equation}
\begin{split}
\label{length of gamma}
&\sum_{j=1}^n d(\gamma(t_{j-1}) , \gamma(t_j))
\\
&\leq \sum_{j=1}^n d(\gamma(t_{j-1}) , \gamma_{i_k}(t_{j-1})) + d(\gamma_{i_k}(t_{j-1}) , \gamma_{i_k}(t_j)) + d(\gamma_{i_k}(t_j) , \gamma(t_j))
\\
&\leq 2 n \eps + \ell_d(\gamma_{i_k})
\leq 2 n \eps + C_Ud(x,y_{i_k}).
\end{split}
\end{equation}
Then we get $\sum_{j=1}^n d(\gamma(t_{j-1}) , \gamma(t_j)) \leq C_U d(x,y)$ by first letting $k \to \infty$ and then $\eps \to 0$. 
Finally we get quasiconvexity by taking the supremum over all partitions of the interval.

To prove the twisted cone condition of uniform curves, we fix $t \in [0,L]$. We want to show that
\begin{equation*}
\min \{ \ell_d(\gamma\vert_{[0,t]}) , \ell_d(\gamma\vert_{[t,L]}) \}
\leq C_U \inf_{z \in \partial X} d(\gamma(t),z).
\end{equation*}
With $\eps>0$ and $K$ as above, let $0 = t_0 < t_1 < ... < t_n = t$ be a partition of the interval $[0,t]$ and let $t = s_0 < s_1 < ... < s_m = L$ be a partition of the interval $[t,L]$. For $k \geq K$, we get as in~\eqref{length of gamma}
\begin{equation*}
\sum_{j=1}^n d(\gamma(t_{j-1}) , \gamma(t_j))
\leq 2 n \eps + \ell_d({\gamma_{i_k}}\vert_{[0,t]})
= 2 n \eps + \min \{t,\ell_d(\gamma_{i_k})\}
\end{equation*}
and similarly
\begin{equation*}
\sum_{j=1}^m d(\gamma(s_{j-1}) , \gamma(s_j))
\leq 2 m \eps + \ell_d({\gamma_{i_k}}\vert_{[t,L]})
= 2 m \eps + (\ell_d(\gamma_{i_k}) - t)_+.
\end{equation*}
By letting $k \to \infty$ first and then $\eps \to 0$ and then taking the supremum over all partitions of the intervals, we get
\begin{equation*}
\min \{ \ell_d(\gamma\vert_{[0,t]}) , \ell_d(\gamma\vert_{[t,L]}) \}
\leq \min\{ \min \{t , \liminf_{k \to \infty} \ell_d(\gamma_{i_k}) \} , ( \liminf_{k \to \infty} \ell_d(\gamma_{i_k}) - t )_+ \}.
\end{equation*}
If $t > \liminf_{k \to \infty} \ell_d(\gamma_{i_k})$, then this is zero and the result clearly holds. Now assume that $0 \leq t \leq \liminf_{k \to \infty} \ell_d(\gamma_{i_k})$. Then
\begin{align*}
\min \{ \ell_d(\gamma\vert_{[0,t]}) , \ell_d(\gamma\vert_{[t,L]}) \}
&\leq \liminf_{k \to \infty} \min \{ t , \ell_d(\gamma_{i_k}) - t \}
\\
&\leq \liminf_{k \to \infty} C_U \inf_{z \in \partial X} d(\gamma_{i_k}(t),z)
= C_U \inf_{z \in \partial X} d(\gamma(t),z).
\end{align*}
This completes the proof for $x\in X$ and $y\in \partial X$.

If both $x,y\in \partial X$, then we obtain a sequence $(x_i)_i$ from $X$ converging to $x$, and for each $i$ we connect $x_i$ to $y$ via a $C_U$-uniform curve in $(X,d)$ and repeat the above proof to this sequence of $C_U$-uniform curves.
Thus we obtain a $C_U$-uniform curve in $(X,d)$ that connects $x$ to $y$, completing the proof.
\end{proof}

\subsection{Basic properties of the sphericalized space}
\label{subs:basic} 

Since we are interested in transformations that make $(X,d_\rho)$  bounded, it is to be expected that the completion of $X$ with respect to $d_\rho$ would result in new points that were not part of the completion of $X$ with respect to $d$; such points should come from sequences that tend to infinity. The goal of this subsection is to demonstrate that when $(X,d)$ is uniform and $\rho$ is quasidecreasing, there is exactly one such point, see Theorem~\ref{thm-9}.

\begin{lem}
\label{lem-5}
Let $x,y \in X$ such that $|x| \leq |y|$ and assume that $\gamma$ is a $C_U$-uniform curve in $(X,d)$ connecting $x$ to $y$. Then
\begin{equation*}
\frac{|x|}{1+C_U}
\leq |z|
\leq (1+C_U)|x| + C_U |y|
\end{equation*}
for every $z \in \gamma$.
\end{lem}

\begin{proof}
Since $d(x,z)\ge |x|-|z|$ and $d(y,z)\ge |y|-|z|$, we see that
\begin{align*}
(1+C_U)|z|
&\geq |z| + C_U d_X(z)
\geq |z| + \min \{ \ell_d(\gamma_{xz}) , \ell_d(\gamma_{zy}) \}
\\
&\geq |z| + \min \{ d(x,z) , d(z,y) \}
\geq |z| + \min \{ |x| - |z| , |y| - |z| \}
= |x|.
\end{align*}
Thus the first inequality holds.
The second inequality holds because by the quasiconvexity of $\gamma$,
\begin{equation*}
|z|
\leq |x| + d(x,z)
\leq |x| + \ell_d(\gamma_{xz})
\leq |x| + C_U\, d(x,y)
\leq (1+C_U)\, |x| + C_U |y|.
\qedhere
\end{equation*}
\end{proof}

In the following theorem, we show that in completing $X$ with respect to the metric $d_\rho$, we have exactly one more element more than in $\overline{X}^d$; this extra point will be denoted by $\infty$ to reflect the fact that this point is the limit of any sequence $(x_i)_{i=1}^{\infty}$, where $\lim_{i \to \infty} |x_i| = \infty$.
The next theorem forms part of Theorem \ref{maintheorem}(a).

\begin{thm}
\label{thm-9}
Let $(X,d)$ be a uniform space and assume that $\rho$ is quasidecreasing and bounded.
Then the space $(X,d_{\rho})$ is bounded, the set $\overline{X}^d$ is a subset of $\overline{X}^{d_\rho}$, and $\overline{X}^{d_\rho}\setminus \overline{X}^d$ contains exactly one point.
\end{thm}

\begin{proof}
First let $x , y \in X$ such that $|x| \leq |y| \leq 2|x|$. We connect $x$ to $y$ with a uniform curve $\gamma$ and thus we get from Lemma~\ref{lem-5} and quasidecreasingness of $\rho$ that
\begin{align*}
d_{\rho}(x,y)
&\leq \int_{\gamma} \rho(|\cdot|)ds
\leq \sup_{t \geq \frac{|x|}{1+C_U}}\rho(t) \ell_d(\gamma)
\leq \sup_{t \geq \frac{|x|}{1+C_U}}\rho(t) C_U d(x,y)
\\
&\leq 6 C_U (1+C_U) \sup_{t \geq \frac{|x|}{1+C_U}}\rho(t) \frac{|x|}{2(1+C_U)}
\leq 6\cqd C_U (1+C_U) \int_{\frac{|x|}{2(1+C_U)}}^{\frac{|x|}{1+C_U}} \rho(t) dt.
\end{align*}

Now let $x , y \in X$ such that $|x| \leq |y|$. Let $k$ be the positive integer such that $2^{k-1} |x| \leq |y| < 2^k |x|$. Then for every $1 \leq i \leq k-1$ let $x_i \in X$ such that $|x_i| = 2^i |x|$. Define also $x_0=x$ and $x_k=y$. Then from the above result and triangle inequality we get
\begin{align*}
d_{\rho}(x,y)
&\leq \sum_{i=1}^k d_{\rho}(x_{i-1},x_i)
\leq 6 \cqd C_U (1+C_U) \sum_{i=1}^k \int_{\frac{|x_{i-1}|}{2(1+C_U)}}^{\frac{|x_{i-1}|}{1+C_U}} \rho(t) dt
\\
&= 6 \cqd C_U (1+C_U) \int_{\frac{|x|}{2(1+C_U)}}^{\frac{2^{k-1}|x|}{1+C_U}} \rho(t) dt
\le 6 \cqd C_U (1+C_U) \int_{\frac{|x|}{2(1+C_U)}}^\infty \rho(t)\, dt
\\
&\leq 6 \cqd C_U (1+C_U) \int_0^{\infty} \rho(t) dt.
\end{align*}
Therefore $(X,d_{\rho})$ is bounded.

For the second claim let $x \in \overline{X}^d$. Then there exists a $d$-Cauchy sequence $(x_i)_{i=1}^{\infty}$ in $X$ that converges to $x$. Then by the quasiconvexity of $(X,d)$ we have $d_{\rho}(x_i,x_j) \leq \sup_{0 < r < \infty} \rho(r) C_U d(x_i,x_j)$ and therefore $(x_i)_{i=1}^{\infty}$ is also $d_{\rho}$-Cauchy and it converges to the same point.

Finally let us prove the third claim. Let $(x_i)_{i=1}^{\infty}$ be a sequence in $X$. We claim that the sequence is $d_{\rho}$-Cauchy and not $d$-Cauchy if and only if $\lim_{i \to \infty} |x_i| =~\infty$. From the penultimate inequality of the above series of inequalities, we know that $d_{\rho}(x_i,x_j) \simle \int_{|x_i|/2(1+C_U)}^{\infty} \rho(t) dt$ whenever $|x_i| \leq |x_j|$.
Thus if $\lim_{i \to \infty} |x_i| = \infty$, then the sequence is $d_{\rho}$-Cauchy, but not $d$-Cauchy.
This also shows that any two sequences $(x_i)_{i=1}^{\infty}$ and $(y_i)_{i=1}^{\infty}$, such that $\lim_{i \to \infty} |x_i| = \lim_{i \to \infty} |y_i| = \infty$, are equivalent, with respect to the metric $d_{\rho}$, and converge to the same point.
This point is the point $\infty$.

If the sequence $(x_i)_{i=1}^{\infty}$ is $d$-bounded, then $|x_i| \leq N < \infty$ for every $i$ and thus we get similar to the proof of Lemma~\ref{lem-6}(i)
\begin{equation*}
d_{\rho}(x_i,x_j)
\geq d(x_i,x_j) \inf_{0 < r \leq |x_i| + d(x_i,x_j)} \rho(r)
\geq d(x_i,x_j) \inf_{0 < r \leq 3N} \rho(r).
\end{equation*}
Because $\rho$ is quasidecreasing, $\inf_{0 < r \leq 3N} \rho(r)$ is positive and does not depend on $i$ or $j$.
Thus if the sequence is $d_{\rho}$-Cauchy, then it is also $d$-Cauchy.

Now assume that the sequence is $d$-unbounded, which implies that the sequence is not $d$-Cauchy. If $\lim_{i \to \infty} |x_i| = \infty$ does not hold, then there exists $N < \infty$ such that for every $n \in \mathbb{N}$ there exists $i \geq n$ such that $|x_i| \leq N$. Let $\eps = \int_N^{2N} \rho(t) dt > 0$.
If the sequence is $d_{\rho}$-Cauchy, there exists $n_{\eps}$ such that $d_{\rho}(x_i,x_j) < \eps$ whenever $i , j \geq n_{\eps}$. Since the sequence is $d$-unbounded, there exists $j$ such that $|x_j| > \max \{ |x_1| , |x_2| , ... , |x_{n_{\eps}}| , 2N \}$, which implies that $j > n_{\eps}$. 
Also there exists $i \geq n_{\eps}$ such that $|x_i| \leq N$. Then we get from Lemma \ref{lem-2} that
\begin{equation*}
\eps
> d_{\rho}(x_i,x_j)
\geq \int_{|x_i|}^{|x_j|} \rho(t) dt
\geq \int_N^{2N} \rho(t) dt
= \eps,
\end{equation*}
which is a contradiction. Thus the sequence is not $d_{\rho}$-Cauchy and we have proved the claim.
\end{proof}

\subsection{Sharpness of Conditions \texorpdfstring{\ref{rhodoubling}}{} and \texorpdfstring{\ref{upperbound}}{}}
\label{subsection:AB}

In this subsection we show the sharpness of Conditions \ref{rhodoubling} and \ref{upperbound} in preserving  uniformity in sphericalization.
More precisely, we show that all bounded quasidecreasing functions $\rho$ that preserve uniformity in the sphericalization of the Euclidean open half-plane also necessarily satisfy Conditions~\ref{rhodoubling} and \ref{upperbound}. It is easy to see that a similar phenomenon happens also in higher-dimensional half-spaces $\mathbb R^n_+$.
Note that we continue to assume that $\rho$ is a positive Borel measurable function and that $\int_0^\infty\rho(t)\, dt$ is finite in this subsection, even though we don't assume that Conditions~\ref{rhodoubling} and \ref{upperbound} hold.

\begin{prop}
Let $X=\mathbb{R} \times (0,\infty)$ be the open upper half-plane with the Euclidean metric and $b=(0,0)$. 
Assume that $\rho$ is bounded and quasidecreasing. Assume also that
\begin{equation}
\label{eq-inf}
\inf_{r > 0} \frac{\int_r^{\infty} \rho(t) dt}{(r+1) \rho(r)} = 0,
\end{equation}
which, by Remark \ref{remark-1}(i), means that $\rho$ does not satisfy Condition \ref{rhodoubling}.
Then the space $(X,|\cdot|)$ is $\frac{\pi}{2}$-uniform, but the space $(X,d_{\rho})$ is not uniform.
\end{prop}

\begin{proof}
We prove via an argument by contradiction that $(X,d_\rho)$ is not a uniform space. Suppose that $(X,d_\rho)$ is a uniform space with constant $C\ge 1$.
If $r$ is a positive number such that $r\le r_0$, then
\[
\frac{\int_{r}^{\infty} \rho(t) dt}{(r+1) \rho(r)} \geq \frac{\int_{r_0}^{\infty} \rho(t) dt}{(r_0+1) \sup_{r' > 0} \rho(r')} > 0.
\]
Thus we have from~\eqref{eq-inf} that for each $r_0>0$, 
\[
\inf_{r \ge r_0} \frac{\int_r^{\infty} \rho(t) dt}{(r+1) \rho(r)} = 0.
\]
From the continuity of the integral we get that there exists $1 < M < \infty$ such that
\begin{equation*}
\int_1^{\infty} \rho(t) dt
= (2C+1) \int_M^{\infty} \rho(t) dt.
\end{equation*}
Thus from our assumption \eqref{eq-inf} it follows that there exists $R>M$ such that
\begin{equation*}
\int_R^{\infty} \rho(t) dt
\leq \frac{(R+1) \rho(R)}{8 C\cqd (C+1)}.
\end{equation*}
Recall that $\cqd \geq 1$ is the constant related to the quasidecreasing property of $\rho$, see~\eqref{eq:QDineq}.
Since $R>M$, by the choice of $M$ we have that $(2C+1)\int_R^\infty \rho(t)\, dt<\int_1^\infty\rho(t)\, dt$.
By the continuity of the integral again, we can find $1<R'<R$ such that
\[
\int_{R'}^\infty\rho(t)\, dt=\frac{2C+1}{\tfrac32C+1}\, \int_R^\infty\rho(t)\, dt,
\]
and then as $(\tfrac32C+1)\int_{R'}^\infty\rho(t)\, dt<\int_1^\infty\rho(t)\, dt$, we can find $1<r<R'$ so that
\[
\int_r^\infty\rho(t)\, dt=(\tfrac32C+1)\int_{R'}^\infty\rho(t)\, dt=(2C+1)\int_R^\infty\rho(t)\, dt.
\]
This means that
\begin{equation*}
\frac{(R+1) \rho(R)}{8 C \cqd (C+1)}
\geq \int_R^{\infty} \rho(t) dt
= \frac{1}{2C} \int_r^R \rho(t) dt
\geq \frac{R-r}{2 C \cqd} \rho(R),
\end{equation*}
and therefore $\frac{r+1}{R+1} \geq \frac{4C + 3}{4C + 4}$.

Now let $x,y \in X$ such that $|x| = |y| = r$ and $d(x,y)=r+1$. 
Then we have for any $r' > r$
\begin{equation*}
d_{\rho}(x,y)
\leq \int_r^{r'} \rho(t) dt + \pi r' \rho(r') + \int_r^{r'} \rho(t) dt.
\end{equation*}
By letting $r' \to \infty$ we get $d_{\rho}(x,y) \leq 2 \int_r^{\infty} \rho(t) dt$, where we used Remark \ref{remark-1}(iii).
Essentially, $d_{\rho}(x,y)$ is at most equal to the path integral of $\rho$ over a curve that goes from $x$ to infinity and then returns to $y$.

Let $\gamma$ be a $C$-uniform curve in $(X,d_{\rho})$ connecting $x$ to $y$, and let $z \in \gamma$. Then by Lemma~\ref{lem-2} we have for any $r' > |z|$,
\begin{align*}
\int_r^{|z|} \rho(t) dt
&\leq \min \{ d_{\rho}(x,z) , d_{\rho}(y,z) \}
\leq \min \{ \ell_{\rho}(\gamma_{xz}) , \ell_{\rho}(\gamma_{zy}) \}
\leq C d_{X,\rho}(z)
\\
&\leq C \left( \int_{|z|}^{r'} \rho(t) dt + \frac{\pi}{2} r' \rho(r') \right)
\xrightarrow[r' \to \infty]{} C \int_{|z|}^{\infty} \rho(t) dt.
\end{align*}
Intuitively, if $|z|$ is large, then $d_{X,\rho}(z)$ is small.
If $|z| \geq R'$, then
\begin{equation*}
\int_r^{\infty} \rho(t) dt
\leq (C+1) \int_{|z|}^{\infty} \rho(t) dt
\leq (C+1) \int_{R'}^{\infty} \rho(t) dt
= \frac{C+1}{\frac32 C + 1} \int_r^{\infty} \rho(t) dt,
\end{equation*}
which is a contradiction.
This is because if $|z|$ is large for some $z \in \gamma$, then $\gamma$ does not satisfy the twisted cone condition.

Therefore we must have $|z| \leq R'$ for every $z \in \gamma$.
As $\gamma$ is also a rectifiable curve with respect to the metric $d$ by Proposition~\ref{rectifiability}, and $\gamma$ is compact in $(X,d)$, there exists $\delta > 0$ such that $|z| \geq \delta$ for every $z \in \gamma$.
Thus we have by Lemma~\ref{lem-10},
\begin{align*}
\int_r^{\infty} \rho(t) dt
&\geq \frac{1}{2} d_{\rho}(x,y)
\geq \frac{1}{2C} \ell_{\rho}(\gamma)
\geq \frac{1}{2C} d(x,y) \inf_{\delta/2 \leq r' \leq R} \rho(r')
\geq \frac{r+1}{2C \cqd} \rho(R)
\\
&\geq 4(C+1) \frac{r+1}{R+1} \int_R^{\infty} \rho(t) dt
\geq 4(C+1) \cdot \frac{4C + 3}{4C + 4} \cdot \frac{1}{2C + 1} \int_r^{\infty} \rho(t) dt
\\
&= \frac{4C + 3}{2C + 1} \int_r^{\infty} \rho(t) dt,
\end{align*}
which is also a contradiction.
In other words, if $|z|$ is relatively small for all $z \in \gamma$, then $\gamma$ is not quasiconvex, because $\ell_{\rho}(\gamma)$ is much larger than $d_{\rho}(x,y)$.

Because $C$ was arbitrary, we conclude that $(X,d_{\rho})$ is not uniform with any constant.
\end{proof}

\begin{prop}
Let $X$ be the open upper half-plane $\mathbb{R} \times (0,\infty)$ with the Euclidean metric and $b=(0,0)$.
Suppose that for every $M < \infty$ we have $\inf_{0 < r \leq M} \rho(r) > 0$. Assume that
\begin{equation}
\label{eq-sup}
\sup_{r > 0} \frac{\int_r^{\infty} \rho(t) dt}{(r+1) \rho(r)} = \infty,
\end{equation}
which means that $\rho$ does not satisfy Condition \ref{upperbound}.
Then the space $(X,|\cdot|)$ is $\frac{\pi}{2}$-uniform, but the space $(X,d_{\rho})$ is not uniform.
\end{prop}

Notice that the condition $\inf_{0 < r \leq M} \rho(r) > 0$ for every $M < \infty$ holds, if we assume that $\rho$ is quasidecreasing.

\begin{proof}
Notice that $d_{\rho}(x,y) \leq \pi |x| \rho(|x|) + \int_{|x|}^{|y|} \rho(t) dt < \infty$ whenever $|x| \leq |y|$, and thus $d_{\rho}$ is a metric in $X$ by Lemma~\ref{lem-6}.
We prove via an argument by contradiction that $(X,d_\rho)$ is not a uniform space.
Suppose that $(X,d_\rho)$ is a uniform space with constant $C\ge 1$.
From the continuity of the integral we get that there exists $M < \infty$ such that
\begin{equation*}
\int_0^{\infty} \rho(t) dt
= 2 \int_M^{\infty} \rho(t) dt.
\end{equation*}
If $0 < r \leq M$, then
\begin{equation*}
\frac{\int_{r}^{\infty} \rho(t) dt}{(r+1) \rho(r)}
\leq \frac{\int_0^{\infty} \rho(t) dt}{\inf_{0 < r \leq M} \rho(r)}
< \infty.
\end{equation*}
Then from our assumption \eqref{eq-sup} it follows that there exists $r > M$ such that
\begin{equation*}
\int_r^{\infty} \rho(t) dt
\geq \pi C (r+1) \rho(r).
\end{equation*}
Again from the continuity of the integral we get that there exist $0 < R_1 < r < R_2$ such that
\begin{equation*}
\int_{R_1}^{\infty} \rho(t) dt
= 2 \int_r^{\infty} \rho(t) dt
= 4 \int_{R_2}^{\infty} \rho(t) dt.
\end{equation*}

Now let $x,y \in X$ such that $|x| = R_1$ and $|y| = R_2$.
Let $\gamma$ be a $C$-uniform curve in $(X,d_{\rho})$ connecting $x$ to $y$. Then there exists $z \in \gamma$ such that $|z| = r$.
Then $d_{X,\rho}(z)$ is at most equal to the path integral of $\rho$ over a curve $\gamma'$ that goes from $z$ to the boundary such that $|z'| = r$ for every $z' \in \gamma'$.
Then by Lemma \ref{lem-2} we get
\begin{align*}
C \frac{\pi}{2} r \rho(r)
&\geq C d_{X , \rho}(z)
\geq \min \{ \ell_{\rho}(\gamma_{xz}) , \ell_{\rho}(\gamma_{zy}) \}
\geq \min \{ d_{\rho}(x,z) , d_{\rho}(y,z) \}
\\
&\geq \min \left\{ \int_{R_1}^r \rho(t) dt , \int_r^{R_2} \rho(t) dt \right\}
= \frac{1}{2} \int_r^{\infty} \rho(t) dt
\geq \frac{\pi}{2} C (r+1) \rho(r).
\end{align*}
This is a contradiction.
Essentially, by choosing $|x|$ small enough and $|y|$ large enough, for any $\gamma$ connecting $x$ and $y$ we find $z \in \gamma$ such that $|z|$ is between $|x|$ and $|y|$ and $d_{X,\rho}(z)$ is significantly smaller than $d_{\rho}(x,z)$ and $d_{\rho}(y,z)$. This means that $\gamma$ does not satisfy the twisted cone condition.

Because $C$ was arbitrary, we conclude that $(X,d_{\rho})$ is not uniform with any constant.
\end{proof}

Assume that $\rho$ is bounded and quasidecreasing. Then the two previous propositions show us that in order to preserve uniformity in the sphericalization of the upper half-plane we must have constants $0 < c \leq C < \infty$ such that
\begin{equation*}
c (r+1) \rho(r)
\leq \int_r^{\infty} \rho(t) dt
\leq C (r+1) \rho(r)
\end{equation*}
for every $r > 0$. Then Lemma \ref{lem-lowerbound} tells us that $\rho$ satisfies Conditions~\ref{rhodoubling} and \ref{upperbound}.

\subsection{Preserving uniformity}
\label{subsect-uniformity}

In the previous subsection we showed that all bounded quasidecreasing functions $\rho$ that preserve uniformity in the sphericalization of the half-plane also necessarily satisfy Conditions~\ref{rhodoubling} and \ref{upperbound}.
Thus, for the rest of the paper we make Borel measurability of $\rho$, finiteness of $\int_0^\infty\rho(t)\, dt$, and both Conditions~\ref{rhodoubling} and \ref{upperbound} standing assumptions. Then we get from Remark~\ref{remark-1} that $\rho$ is bounded and quasidecreasing.
In this subsection we show that $(X,d_{\rho})$ is a uniform space.

\begin{lem}
\label{lem-14}
Let $x , y \in X$ such that $\frac{1}{2} (|x|+1) \leq |y|+1 \leq 2(|x|+1)$. Assume that $\gamma$ is a $C_U$-uniform curve in $(X,d)$ that connects $x$ to $y$. Then
\begin{equation*}
\frac{1}{C_A^{3+\log_2(C_U)}}\, \rho(|x|)
\leq \rho(|z|)
\leq C_A^{3+\log_2(C_U)} \, \rho(|x|)
\end{equation*}
for every $z \in \gamma$.
\end{lem}

\begin{proof}
Note that by~\ref{rhodoubling} we have that $\rho(2^i(r+1)-1)\le C_A\, \rho(2^{i+1}(r+1)-1)$ and $\rho(2^{i+1}(r+1)-1) \leq C_A \rho(2^i(r+1)-1)$, for every $r > 0$ and for each non-negative integer $i$.
We will use these repeatedly below.

From Lemma \ref{lem-5} we get that for every $z \in \gamma$
\begin{equation}
\label{z-estimate}
\begin{split}
\frac{|x|-1}{2(1+C_U)}
&\leq \frac{\min\{|x|,|y|\}}{1+C_U}
\leq |z|
\leq (1+C_U)\min\{|x|,|y|\} + C_U \max\{|x|,|y|\}
\\
&\leq (1+C_U)|x| + C_U (2|x|+1)
= (1+3C_U)|x|+C_U.
\end{split}
\end{equation}

Assume first that $|z| \geq |x|$. Let $K$ be the positive integer such that $2^{K-1} \leq 1+3C_U < 2^K$ and let $k$ be the positive integer such that $2^{k-1} (|x|+1) \leq |z|+1 < 2^k(|x|+1)$. Then by repeatedly using the inequality at the beginning of this proof and then using \ref{rhodoubling} one more time, we get
\begin{equation*}
\rho(|x|)
\leq C_A^{k-1} \rho(2^{k-1}(|x|+1)-1)
\leq C_A^k \rho(|z|).
\end{equation*}
Similarly, we have
\begin{equation*}
\rho(|z|)
\leq C_A \rho(2^{k-1}(|x|+1)-1)
\leq C_A^k \rho(|x|).
\end{equation*}

From \eqref{z-estimate} we have $|z|+1 \leq (1+3C_U)|x|+C_U+1 \leq (1+3C_U)(|x|+1)$ and therefore
\begin{equation*}
2^{k-1}
\leq \frac{|z|+1}{|x|+1}
\leq 1 + 3C_U
< 2^K.
\end{equation*}
From this we conclude that $k-1<K$ i.e. $k \leq K$.
Thus
\begin{equation*}
\rho(|x|)
\leq C_A^k \rho(|z|)
\leq C_A^K \rho(|z|)
\leq C_A^{1+ \log_2(1+3C_U)} \rho(|z|)
\end{equation*}
and $\rho(|z|) \leq C_A^{1+ \log_2(1+3C_U)} \rho(|x|)$.

Now assume that $|z| < |x|$. Let $K$ be the positive integer such that $2^{K-1} \leq 1+C_U < 2^K$. 
Assume first that $|z| \geq |x|/(1+C_U)$ and let $k$ be the positive integer such that $2^{k-1} |z| \leq |x| < 2^k |z|$. 
Then by repeatedly using \eqref{eq:power-doubling} we have
\begin{equation*}
\rho(|z|)
\leq C_A^{k-1} \rho(2^{k-1}|z|)
\leq C_A^k \rho(|x|)
\end{equation*}
and similarly,
\begin{equation*}
\rho(|x|)
\leq C_A \rho(2^{k-1}|z|)
\leq C_A^k \rho(|z|).
\end{equation*}

We also have
\begin{equation*}
2^{k-1}
\leq \frac{|x|}{|z|}
\leq 1+C_U
< 2^K
\end{equation*}
and therefore $k \leq K$. Thus
\begin{equation*}
\rho(|z|)
\leq C_A^k \rho(|x|)
\leq C_A^K \rho(|x|)
\leq C_A^{1+\log_2(1+C_U)} \rho(|x|)
\end{equation*}
and $\rho(|x|) \leq C_A^{1+\log_2(1+C_U)} \rho(|z|)$.

Finally assume that $|z| < |x|/(1+C_U)$. In this case we have by \eqref{z-estimate}
\begin{equation*}
|z|
< \frac{|x|}{1+C_U}
\leq 2 \frac{|x|-1}{2(1+C_U)}+1
\leq 2|z|+1
\end{equation*}
and thus from \ref{rhodoubling} and the above we have
\begin{equation*}
\rho(|z|)
\leq C_A \rho \left( \frac{|x|}{1+C_U} \right)
\leq C_A^{2+\log_2(1+C_U)} \rho(|x|)
\end{equation*}
and similarly $\rho(|x|) \leq C_A^{2+\log_2(1+C_U)} \rho(|z|)$.

In conclusion the claim holds with the comparison constant $C_A^{3+\log_2(C_U)}$, because $3 + \log_2(C_U) \geq 1 + \log_2(1+3C_U) \geq 2 + \log_2(1+C_U) > 1 + \log_2(1+C_U)$.
\end{proof}

\begin{lem}
\label{lem-3}
Suppose that $(X,d)$ is $C_U$-uniform. 
Whenever $x,y \in X$ satisfy $\frac{1}{2}(|x|+1) \leq |y|+1 \leq 2(|x| + 1)$, then
\begin{equation*}
\frac{1}{C_A^3\, C_B} \rho(|x|) d(x,y)
\leq d_{\rho}(x,y)
\leq C_2 \rho(|x|) d(x,y),
\end{equation*}
where $C_2 = C_U C_A^{3+\log_2(C_U)}$.
\end{lem}

\begin{proof}
We have $d(x,y) \leq |x| + |y| \leq 3|x| + 1$.
By Remark~\ref{remark-1}(ii), we know that $\rho$ is quasidecreasing with constant $\cqd=C_AC_B$.
So, by considering curves connecting $x$ to $y$ and using the definition of $d_\rho$, we get
\begin{align*}
d_{\rho}(x,y)
&\geq d(x,y) \inf_{0 < r \leq |x| + d(x,y)} \rho(r)
\geq d(x,y) \frac{1}{C_AC_B} \rho(4|x|+1)
\\
&\geq \frac{1}{C_A^3C_B} \rho(|x|) d(x,y).
\end{align*}
In obtaining the last inequality above, we used~\ref{rhodoubling} twice.

To prove the second inequality we consider a $C_U$-uniform curve $\gamma$ that connects $x$ to $y$. 
Then from Lemma~\ref{lem-14} we get
\begin{equation*}
d_{\rho}(x,y)
\leq \int_{\gamma} \rho(|\cdot|)ds
\leq C_A^{3+\log_2(C_U)} \rho(|x|) \ell_d(\gamma)
\leq C_U C_A^{3+\log_2(C_U)} \rho(|x|) d(x,y).
\qedhere
\end{equation*}
\end{proof}

\begin{lem}
\label{lem-4}
Suppose that $(X,d)$ is $C_U$-uniform.
Let $x,y \in X$ such that $|y| \geq 2 |x| + 1$. Then
\begin{equation*}
\frac{1}{C_A} \rho(|x|)(|x|+1)
\leq
d_{\rho}(x,y) \le 3C_U\, C_A^{4+\log_2(C_U)}\, C_B\, 
\rho(|x|) (|x|+1).
\end{equation*}
In particular,
\begin{equation*}
\frac{1}{C_A}\, \rho(|x|)\, (|x|+1)\le  d_{\rho}(x,\infty)\le 3C_U\, C_A^{4+\log_2(C_U)}\, C_B\, 
\rho(|x|) (|x|+1),
\end{equation*}
where $\infty$ is the point defined right before Theorem \ref{thm-9}.
\end{lem}

\begin{proof}
We get from Lemma \ref{lem-2} and Remark \ref{remark-1}(i) that $d_{\rho}(x,y) \geq \frac{1}{C_A} \rho(|x|) (|x|+1)$.
Let $k$ be the positive integer such that $2^k |x| < |y| \leq 2^{k+1} |x|$. For every $1 \leq i \leq k$ let $x_i \in X$ such that $|x_i| = 2^i |x|$. 
Define also $x_0=x$ and $x_{k+1} = y$. Then from the triangle inequality and Lemma \ref{lem-3} we get
\begin{align*}
d_{\rho}(x,y)
&\leq \sum_{i=1}^{k+1} d_{\rho}(x_i,x_{i-1})
\leq \sum_{i=1}^{k+1} C_2 \rho(|x_{i-1}|) d(x_i,x_{i-1})
\\
&\leq \sum_{i=1}^{k+1} C_2 \rho(2^{i-1} |x|) 3 \cdot 2^{i-1} |x|
\leq 3 C_2 C_A \int_{|x|}^{2^{k+1} |x|} \rho(t) dt
\\
&\leq 3 C_2 C_A C_B \rho(|x|)(|x|+1),
\end{align*}
where we also used Conditions \ref{rhodoubling} and \ref{upperbound}.
As $C_2=C_U\, C_A^{3+\log_2(C_U)}$, we obtain the desired estimate.
\end{proof}

\begin{lem}
\label{lem-11}
Let $x,y \in X$ such that $2^k |x| \leq |y|$ for some $k \geq 1$.
Let $\gamma$ be a $C_U$-uniform curve in $(X,d)$ connecting $x$ to $y$ and assume that $\gamma$ is parametrized by arc-length with respect to the metric $d$. 
For every $0 \leq i \leq k-1$ we define the number $0 \leq t_i < \ell_d(\gamma)$ such that the point $x_i := \gamma(t_i)$ is the first point on the curve, starting from $x$, such that $|x_i| = 2^i |x|$. We also set $t_k=\ell_d(\gamma)$ and $x_k = y$. Then for every $1 \leq i \leq k$ the subcurve $\gamma\vert_{[t_{i-1},t_i]}$ is $C_U(2C_U+3)$-uniform in $(X,d)$.
\end{lem}

\begin{proof}
The twisted cone condition is clearly satisfied, because for every $t_{i-1} \leq t \leq t_i$ we have
\begin{equation*}
\min \{ \ell_d(\gamma\vert_{[t_{i-1} , t]}) , \ell_d(\gamma\vert_{[t,t_i]}) \}
\leq \min \{ \ell_d(\gamma\vert_{[0,t]}) , \ell_d(\gamma\vert_{[t,\ell_d(\gamma)]}) \}
\leq C_U d_X(\gamma(t)),
\end{equation*}
since $\gamma$ is $C_U$-uniform.

For quasiconvexity assume first that $\ell_d(\gamma\vert_{[0,t_i]}) \leq \ell_d(\gamma\vert_{[t_i,\ell_d(\gamma)]})$. Then $i \neq k$ and thus
\begin{equation*}
\ell_d(\gamma\vert_{[t_{i-1},t_i]})
\leq \ell_d(\gamma\vert_{[0,t_i]})
\leq C_U d_X(x_i)
\leq C_U |x_i|
\leq 2C_U d(x_{i-1},x_i),
\end{equation*}
so the curve is $2C_U$-quasiconvex.

Now assume that $\ell_d(\gamma\vert_{[0,t_i]}) > \ell_d(\gamma\vert_{[t_i,\ell_d(\gamma)]})$ and $i \neq k$. Then
\begin{equation*}
C_U |x_i|
\geq C_U d_X(x_i)
\geq \ell_d(\gamma\vert_{[t_i,\ell_d(\gamma)]})
\geq d(x_i,y)
\geq |y| - |x_i|
\end{equation*}
and therefore $|y| \leq (1+C_U)|x_i|$. Then we have
\begin{align*}
\ell_d(\gamma\vert_{[t_{i-1},t_i]})
&\leq \ell_d(\gamma)
\leq C_U d(x,y)
\leq C_U (|x|+|y|)
\leq C_U (|x|+(1+C_U)|x_i|)
\\
&= C_U (|x| + (1+C_U) 2^i |x|)
\leq C_U(2C_U+3) d(x_{i-1},x_i).
\end{align*}

Finally if $i = k$, we have $|y| \leq d(x_{k-1},x_k) + |x_{k-1}| \leq 2 d(x_{k-1},x_k)$ and therefore
\begin{equation*}
\ell_d(\gamma\vert_{[t_{k-1},t_k]})
\leq C_U (|x|+|y|)
\leq 2 C_U |y|
\leq 4 C_U d(x_{k-1},x_k).
\qedhere
\end{equation*}

\end{proof}

\begin{lem}
\label{lem-12}
Let $x , y \in X$ such that $|x| \leq |y|$ and let $\gamma$ be a $C_U$-uniform curve in $(X,d)$ connecting $x$ to $y$. 
Suppose that $\gamma$ is parametrized by arc-length with respect to the metric $d$. 
Then for every $t$ and $t'$ with $0 \leq t \leq t' \leq \ell_d(\gamma)$ we have $\rho(|\gamma(t')|) \simle \rho(|\gamma(t)|)$ with the comparison constant depending on $C_U$, $C_A$ and $C_B$.
\end{lem}

\begin{proof}
If $|y| \leq 2|x|$, then by Lemma \ref{lem-14} we have $\rho(|z|) \simeq \rho(|x|)$ for every $z \in \gamma$ and so the claim holds.
If $|y| \geq 2|x|$, then let $k$ be the positive integer such that $2^k|x| \leq |y| < 2^{k+1}|x|$. We divide $\gamma$ into subcurves $\gamma\vert_{[t_{i-1},t_i]}$ the same way as in Lemma \ref{lem-11}. Assume that $t_{i-1} \leq t \leq t_i$ and $t_{j-1} \leq t' \leq t_j$ with $1 \leq i \leq j \leq k$.
According to Lemma \ref{lem-11} the curve $\gamma\vert_{[t_{i-1},t_i]}$ is $C_U(2C_U+3)$-uniform for every $1 \leq i \leq k$. 
Therefore we get from Lemma \ref{lem-14} that $\rho(|\gamma(t)|) \simeq \rho(2^i |x|)$ and similarly $\rho(|\gamma(t')|) \simeq \rho(2^j |x|)$.
Finally from the quasidecreasingness of $\rho$ we get $\rho(|\gamma(t')|) \simeq \rho(2^j|x|) \simle \rho(2^i|x|) \simeq \rho(|\gamma(t)|)$.
\end{proof}

Now we are ready to prove the main result of this section, which also completes the proof of Theorem~\ref{maintheorem}(a).

\begin{thm}
\label{thm-uniformity}
Assume that $(X,d)$ is a $C_U$-uniform space and Conditions \ref{rhodoubling} and \ref{upperbound} hold.
Then $(X,d_{\rho})$ is a uniform space with uniformity constant $C_{U,\rho}$ that depends only on $C_U$, $C_A$ and $C_B$.
Also any point $x \in X$ can be connected to the point $\infty$ with a curve that is $C_{U,\rho}$-uniform in the space $(X,d_{\rho})$.
\end{thm}

\begin{proof}
Note that by Theorem \ref{thm-9} we have $\partial_{\rho} X = \partial X \cup \{ \infty \}$ and therefore $(X,d_{\rho})$ is not complete.
From Proposition \ref{homeo} we get that $(X,d_{\rho})$ is locally compact because $(X,d)$ is locally compact.
What is left to show is that any two points $x , y \in X$ can be connected by a curve $\gamma$ that is $C_{U,\rho}$-uniform in $(X,d_{\rho})$. Then we get the final claim from Theorem \ref{thm-extension}.

Let us assume by symmetry that $|x| \leq |y|$. Let $\gamma$ be a $C_U$-uniform curve in $(X,d)$ connecting $x$ to $y$. Without loss of generality we may assume that $\gamma$ does not cross itself. From Propositions \ref{homeo} and \ref{rectifiability} we get that $\gamma$ is also a curve in $(X,d_{\rho})$. We shall show that $\gamma$ is uniform in $(X,d_{\rho})$ with uniformity constant depending only on $C_U$, $C_A$ and $C_B$.

From Remark \ref{remark-1}(ii) we get that $\rho$ is quasidecreasing.
If $|y| \geq 2|x| + 1$, let $k$ be the positive integer such that $2^k |x| \leq |y| < 2^{k+1}|x|$. Similar to Lemma \ref{lem-11}, we define $x_0=x$, $x_k=y$ and for every $1 \leq i \leq k-1$ the point $x_i \in \gamma$ as the first point on the curve such that $|x_i| = 2^i |x|$.
Let us prove first that the curve is quasiconvex.

If $|y| \leq 2|x| + 1$, then from Lemma \ref{lem-14} we get that $\rho(|z|) \simeq \rho(|x|)$ for every $z \in \gamma$. Thus by \eqref{rholength}, the uniformity of $\gamma$ and Lemma \ref{lem-3}
\begin{equation*}
\ell_{\rho}(\gamma)
\leq \int_{\gamma} \rho(|\cdot|)ds
\simeq \rho(|x|) \ell_d(\gamma)
\simeq \rho(|x|) d(x,y)
\simeq d_{\rho}(x,y).
\end{equation*}

Now assume that $|y| \geq 2|x| + 1$.
According to Lemma \ref{lem-11} the curve $\gamma_{x_{i-1} x_i}$ is $C_U(2C_U+3)$-uniform for every $1 \leq i \leq k$. 
Therefore we get from Lemma \ref{lem-14} that $\rho(|z|) \simeq \rho(2^i |x|)$ for every $z \in \gamma_{x_{i-1} x_i}$.
Therefore
\begin{align*}
\ell_{\rho}(\gamma)
&\leq \int_{\gamma} \rho(|\cdot|)ds
= \sum_{i=1}^k \int_{\gamma_{x_{i-1} x_i}} \rho(|\cdot|)ds
\simeq \sum_{i=1}^k \rho(2^i |x|) \ell_d(\gamma_{x_{i-1} x_i})
\\
&\simeq \sum_{i=1}^k \rho(2^i |x|) d(x_{i-1},x_i)
\simeq \sum_{i=1}^k \rho(2^i |x|) 2^{i-1} |x|
\leq C_A \int_{|x|}^{2^k |x|} \rho(t) dt
\\
&\leq C_A C_B (|x|+1) \rho(|x|)
\simeq d_{\rho}(x,y).
\end{align*}
Here we used the property that $\gamma_{x_{i-1} x_i}$ is $C_U(2C_U+3)$-quasiconvex alongside~\eqref{rholength}, Conditions~\ref{rhodoubling} and~\ref{upperbound}, and Lemma~\ref{lem-4}.
Thus we have shown the quasiconvexity.

Now we prove the twisted cone condition. Let $z \in \gamma$ such that $y \neq z \neq x$. Then $z \neq \infty$ and also $d_X(z) > 0$ and therefore $d_{X , \rho}(z) > 0$. Let $z' \in \partial_{\rho} X$ such that $d_{\rho}(z,z') \simeq d_{X , \rho}(z)$.

First we make some estimates for $\ell_{\rho}(\gamma_{zy})$.
If $|y| \leq 2|x| + 1$, then by Lemma~\ref{lem-14} we have $\rho(|z|) \simeq \rho(|x|)$ and therefore 
\begin{equation*}
\ell_{\rho}(\gamma_{zy})
\leq \ell_{\rho}(\gamma)
\leq \int_{\gamma} \rho(|\cdot|)ds
\simeq \rho(|z|) \ell_d(\gamma)
\simeq \rho(|z|) d(x,y)
\simle \rho(|z|) (|z|+1),
\end{equation*}
where we used \eqref{rholength} and Lemma~\ref{lem-5}.
If $|y| \geq 2|x| + 1$ and we assume that $z \in \gamma_{x_{i-1} x_i}$, then
\begin{align*}
\ell_{\rho}(\gamma_{zy})
&\leq \int_{\gamma_{zy}} \rho(|\cdot|)ds
\leq \int_{\gamma_{x_{i-1} y}} \rho(|\cdot|)ds
= \sum_{j=i}^k \int_{\gamma_{x_{j-1} x_j}} \rho(|\cdot|)ds
\\
&\simeq \sum_{j=i}^k \ell_d(\gamma_{x_{j-1} x_j}) \rho(2^j|x|)
\simle \int_{2^{i-1} |x|}^{2^k |x|} \rho(t) dt
\leq C_B (2^{i-1} |x| + 1) \rho(2^{i-1}|x|)
\\
&\simeq \rho(|z|)(|z|+1),
\end{align*}
where we used Conditions~\ref{rhodoubling},~\ref{upperbound}, \eqref{rholength}, and Lemmas~\ref{lem-5},~\ref{lem-14}, and~\ref{lem-11}.

Now if $|z| > 2 |z'| + 1$, then $d_{\rho}(z,z') \simeq \rho(|z'|) (|z'|+1) \simge \rho(|z|)(|z|+1)$, where we used Lemma~\ref{lem-4} and Remark~\ref{remark-1}(ii).
If $|z'| > 2 |z| + 1$, or $z' = \infty$, then 
\[
d_{\rho}(z,z') \simeq \rho(|z|) (|z|+1), 
\]
where we used Lemma~\ref{lem-4}.
If $|z| \leq 2 |z'| + 1$, $|z'| \leq 2 |z| + 1$ and $d(z,z') \geq \frac{|z|+1}{2C_U}$, then Lemma \ref{lem-3} gives us that 
\[
d_{\rho}(z,z') \simeq \rho(|z|) d(z,z') \simge \rho(|z|) (|z|+1).
\]
In all these cases we get $d_{X , \rho}(z) \simge \rho(|z|)(|z|+1) \simge \ell_{\rho}(\gamma_{zy})$.

Finally let us assume that $|z| \leq 2 |z'| + 1$, $|z'| \leq 2 |z| + 1$, and $d(z,z') \leq \frac{|z|+1}{2C_U}$. In this case Lemma~\ref{lem-3} gives us that $d_{\rho}(z,z') \simeq \rho(|z|) d(z,z')$.
If $\ell_d(\gamma_{xz}) \leq \ell_d(\gamma_{zy})$, then
\begin{equation*}
|z| - |x|
\leq d(x,z)
\leq \ell_d(\gamma_{xz})
\leq C_U d_X(z)
\leq C_U d(z,z')
\leq \frac{|z|+1}{2},
\end{equation*}
and therefore $|z| \leq 2|x|+1$.
This together with Lemma \ref{lem-5} and \ref{rhodoubling} implies that $\rho(|x|) \simeq \rho(|z|)$ and then by~\eqref{rholength} and Lemma \ref{lem-12}
\begin{equation*}
\ell_{\rho}(\gamma_{xz})
\leq \int_{\gamma_{xz}} \rho(|\cdot|)ds
\simle \rho(|x|) \ell_d(\gamma_{xz})
\simle \rho(|z|) d_X(z)
\leq \rho(|z|) d(z,z').
\end{equation*}
Assume now that $\ell_d(\gamma_{xz}) \geq \ell_d(\gamma_{zy})$. Then we get from \eqref{rholength} and Lemma \ref{lem-12} that
\begin{equation*}
\ell_{\rho}(\gamma_{zy})
\leq \int_{\gamma_{zy}} \rho(|\cdot|)ds
\simle \rho(|z|) \ell_d(\gamma_{zy})
\simle \rho(|z|) d_X(z)
\leq \rho(|z|) d(z,z').
\end{equation*}
This completes the proof.
\end{proof}

\section{The doubling property of \texorpdfstring{$\mu_{\rho}$}{}}
\label{sect-doubling}

The focus of this section is to consider the doubling properties of the measure $\mu$ on $(X,d)$ and the measure $\mu_{\rho}$ on $(X,d_\rho)$ as given in Definition \ref{murhodef}.
We point out here that there are \emph{two} parameters involved in this transformation; the density function $\rho$ and the index $\sigma>0$. It turns out that the precise value of $\sigma$ does not interfere with the doubling property of $\mu_\rho$, but does influence the constant associated with the doubling condition, see Theorem~\ref{thm-doubling}.

Throughout this section, $(X,d)$ is assumed to be a $C_U$-uniform space equipped with a Borel regular measure $\mu$ that is \emph{doubling} i.e. there is a constant $C_\mu > 1$ such that
\begin{equation*}
    0 < \mu(B(x,2r)) \le C_\mu \mu(B(x,r)) < \infty
\quad \text{for all $x \in X$ and $r>0$}.
\end{equation*}
By a standard argument this inequality holds also for $x \in \partial X$.
The notation $B$ is used for balls with respect to the metric $d$ while $B_{\rho}$ is used for balls with respect to the metric $d_{\rho}$:
\begin{equation*}
B_\rho(x,r) := \{y\in X: d_\rho(x,y)<r\}
\end{equation*}
for $x \in \overline X^{d_{\rho}}$ and $r>0$.
In addition we assume that Conditions \ref{rhodoubling} and \ref{upperbound} hold.
\begin{deff}
\label{murhodef}
We define a measure $\mu_{\rho}$ to the space $(X,d_\rho)$ by setting
\begin{equation*}
\mu_{\rho}(A):=\int_A \rho(|x|)^\s d\mu(x),
\end{equation*}
for any $\mu$-measurable set $A \subset X$. Here $\s$ is a positive parameter.
\end{deff}
The goal of this section is to show that the measure $\mu_{\rho}$ is doubling in $(X,d_{\rho})$.
We first show the necessity of \ref{condition-C} in obtaining the doubling property of $\mu_\rho$.

\subsection{Preliminary results and necessity of Condition \texorpdfstring{\ref{condition-C}}{}}
\label{sect-necessityC}

In this subsection we first prove some preliminary results that don't require Condition \ref{condition-C}. Then we prove that Condition~\ref{condition-C} is necessary for preserving the doubling property of the measure. In the next subsection we make Condition~\ref{condition-C} a standing assumption and show that $\mu_{\rho}$ is doubling.

To shorten notations let's define $h\colon (0,\infty)\to (0,\infty)$ by setting $h(t):=(t+1)\rho(t)$. By Remark~\ref{remark-1}(ii) $h$ is quasidecreasing with constant $C_A C_B$ and by Remark \ref{remark-1}(i) we have $\lim_{t \to \infty} h(t) = 0$. The inverse of $h$
\[
h^{-1}(\tau):=\inf \{t > 0: h(t)\le \tau\},\quad \tau>0
\]
will be very valuable for us. Notice that $h^{-1}$ is monotone decreasing and unbounded.
For future reference fix $\tau_1 = \rho(1) / C_A$. Thus $h^{-1}(\tau_1)\ge 1$ and, when $\tau\in(0,\tau_1]$, we know that $h^{-1}(\tau)\ge 1>0$.
From Lemma \ref{BaloghLem} we get $h^{-1}(\tau_1) \leq 2(2C_A^2 C_B)^{1 / \eps} - 1$ and therefore $h^{-1}(\tau_1) \simeq 1$.

\begin{lem}\label{h-1doubling}
    For $\tau\in (0,\tau_1]$, we have
    \[
    h^{-1}(\tau)\le h^{-1}(\tau/2)\simle h^{-1}(\tau),
    \]
    where the comparison constant only depends on $C_A$ and $C_B$.
\end{lem}
\begin{proof}
The first inequality is trivial. For the second fix $\tau\in (0,\tau_1]$. Let $\eps$ be as in Lemma \ref{BaloghLem}.
If $2 (2 C_A C_B)^{1/\eps} \geq h^{-1}(\tau/2)$, then the result follows trivially from $h^{-1}(\tau) \geq h^{-1}(\tau_1) \geq 1$.
Therefore assume that $2 (2 C_A C_B)^{1/\eps} \leq h^{-1}(\tau/2)$.
Let $0 < r \leq \frac{1}{2} (2 C_A C_B)^{-1/\eps} h^{-1}(\tau/2)$. Then $(2 C_A C_B)^{1/\eps}(r+1) - 1 < h^{-1}(\tau/2)$ and thus $h((2 C_A C_B)^{1/\eps}(r+1)-1) > \tau/2$. Then with Lemma \ref{BaloghLem} we get $\tau < 2h((2 C_A C_B)^{1/\eps}(r+1)-1) \leq h(r)$. Because this holds for all such $r$, we conclude that $h^{-1}(\tau) \geq \frac{1}{2} (2 C_A C_B)^{-1/\eps} h^{-1}(\tau/2)$.
\end{proof}

The following is a re-statement of the second statement of Lemma~\ref{lem-4} from the point of view of $h$. 

\begin{lem}\label{constants}
    For all $x\in X$,
    \[
    \frac{1}{C_A} h(|x|)\le d_\rho(x,\infty)\le C_1h(|x|),
    \]
    where $C_1 = 3C_U\, C_A^{4+\log_2(C_U)}\, C_B$.
\end{lem}

Balls with respect to the metric $d$ and balls with respect to the metric $d_\rho$ need not be the same. This is especially true for balls in the metric $d_\rho$, centered at the new point $\infty$. The goal of the next lemma is to compare the $d_\rho$-balls centered at $\infty$ with complements of $d$-balls centered at $b$.

\begin{lem}\label{small balls at infty}
    For $0<r\le \tau_1/C_A$ we have
    $$
    X\setminus B \left( b,2h^{-1} \left( \frac{r}{2 C_1 C_A C_B} \right) \right) 
    \subset B_\rho(\infty,r) \subset X\setminus B(b,h^{-1}(C_A r)).
    $$ 
\end{lem}

\begin{proof}
    Note first that $h^{-1}(C_A r)>0$, since $C_A r \le \tau_1$. Similarly, $h^{-1}(\frac{r}{2 C_1 C_A C_B})>0$.
    To prove the first inclusion suppose that there is $y\notin B_\rho(\infty,r)$. Then by Lemma~\ref{constants} and by the quasidecreasing property of $h$, 
    \[
    r\le d_\rho(\infty, y)\le C_1h(|y|)\le C_1 C_A C_B h(t)<2 C_1 C_A C_B h(t)\quad \text{for all }0 < t \le |y|.
    \]
    Thus 
    \[
    h^{-1} \left( \frac{r}{2 C_1 C_A C_B} \right) \ge |y|
    \]
    and the first inclusion is proved.
    
    The proof of the second inclusion is even simpler. It follows directly using Lemma \ref{constants} and the definition of $h^{-1}$.
\end{proof}

Next we show that Condition~\ref{condition-C} is necessary to preserve the doubling property of the measure.

\begin{thm}
Let $(X,d)$ be a $C_U$-uniform space with doubling measure $\mu$ and assume that $\rho$ satisfies Conditions~\ref{rhodoubling} and~\ref{upperbound}.
Assume also that
\begin{equation}
\label{eq-sup-sigma}
\sup_{r > 0} \frac{\int_{X \setminus B(b,r)} \rho(|x|)^{\s} d \mu(x)}{\rho(r)^{\s} \mu(B(b,r+1))} = \infty,
\end{equation}
which means that Condition \ref{condition-C} does not hold.
Then the measure $\mu_{\rho}$ is not doubling in the space $(X,d_{\rho})$.
\end{thm}

\begin{proof}

By Theorem \ref{thm-9} the space $(X,d_{\rho})$ is bounded. If $\int_X \rho(|x|)^{\s} d \mu(x) = \infty$, then $\mu_{\rho}(X) = \infty$ and thus $\mu_{\rho}$ is not doubling. Therefore let us assume that $\int_X \rho(|x|)^{\s} d \mu(x) < \infty$ and assume by contradiction that $\mu_{\rho}$ is doubling in $(X,d_{\rho})$ with doubling constant $C_{\mu_{\rho}}$.
By Lemma \ref{h-1doubling} and Condition \ref{rhodoubling} there exists a constant $D$ depending only on $C_U$, $\s$, $C_A$ and $C_B$ such that
\begin{equation*}
\rho \left( h^{-1}(2C_A r) \right)^{\s}
\leq D \rho(R)^{\s}
\end{equation*}
whenever $h^{-1} \left( \frac{r}{2 C_1 C_A C_B} \right) \leq \frac{R}{2} \leq h^{-1} \left( \frac{r}{4 C_1 C_A C_B} \right)$ and $2 C_A r \leq \tau_1$.
Here the constant $C_1$ is the constant in Lemma \ref{constants} and $\tau_1 = \rho(1) / C_A$.

From Remark \ref{remark-1}(ii) we get that $\rho$ is quasidecreasing with constant $C_A C_B$.
If $R \leq M := 2 h^{-1} \left( \frac{\tau_1}{6 C_1 C_A^2 C_B} \right)$, then
\begin{equation*}
\frac{\int_{X \setminus B(b,R)} \rho(|x|)^{\s} d \mu(x)}{\rho(R)^{\s} \mu(B(b,R+1))}
\leq \frac{(C_A C_B)^{\s} \int_X \rho(|x|)^{\s} d \mu(x)}{\rho(M)^{\s} \mu(B(b,1))}
< \infty.
\end{equation*}
This and \eqref{eq-sup-sigma} imply that there exists $R > M$ such that
\begin{equation*}
\int_{X \setminus B(b,R)} \rho(|x|)^{\s} d \mu(x)
\geq 2 C_{\mu_{\rho}}^3 D (C_A C_B)^{\s} \rho(R)^{\s} \mu(B(b,R+1)).
\end{equation*}

Because $h^{-1}$ is monotone decreasing and unbounded, there exists a positive number $\tau_2$ such that $h^{-1}(\tau) \leq R/2$ whenever $\tau > \tau_2$ and $h^{-1}(\tau) \geq R/2$ whenever $\tau < \tau_2$.
If $\tau < \tau_2$, we get
\begin{equation*}
h^{-1}(\tau)
\geq \frac{R}{2}
> \frac{M}{2}
= h^{-1} \left( 
\frac{\tau_1}{6 C_1 C_A^2 C_B} \right),
\end{equation*}
which implies that $\tau < \frac{\tau_1}{6 C_1 C_A^2 C_B}$. Because this holds for every $\tau < \tau_2$, we conclude that $\tau_2 \leq \frac{\tau_1}{6 C_1 C_A^2 C_B}$.
Let us define $r = 3 C_1 C_A C_B \tau_2$. Then $\frac{r}{4 C_1 C_A C_B} < \tau_2 < \frac{r}{2 C_1 C_A C_B}$ and therefore $h^{-1} \left( \frac{r}{2 C_1 C_A C_B} \right) \leq \frac{R}{2} \leq h^{-1} \left( \frac{r}{4 C_1 C_A C_B} \right)$.
Also we have $2 C_A r = 6 C_1 C_A^2 C_B \tau_2 \leq \tau_1$.
Finally from Lemma \ref{small balls at infty} we get that
\begin{equation*}
B_{\rho}(\infty,2r) \setminus B_{\rho}(\infty,r) \subset B \left( b,2h^{-1} \left( \frac{r}{2 C_1 C_A C_B} \right) \right) \setminus B \left( b,h^{-1} \left( 2 C_A r \right) \right).
\end{equation*}

With these estimates and Lemma \ref{small balls at infty} we get
\begin{align*}
\mu_{\rho}(B_{\rho}(\infty,r))
&\geq \int_{X \setminus B \left( b,2h^{-1} \left( \frac{r}{2 C_1 C_A C_B} \right) \right) } \rho(|x|)^{\s} d \mu(x)
\\
&\geq \int_{X \setminus B(b,R)} \rho(|x|)^{\s} d \mu(x)
\geq 2 C_{\mu_{\rho}}^3 D (C_A C_B)^{\s} \rho(R)^{\s} \mu(B(b,R+1))
\\
&\geq 2 C_{\mu_{\rho}}^3 (C_A C_B)^{\s} \rho \left( h^{-1}(2C_A r) \right)^{\s} \mu \left( B \left( b,2h^{-1} \left( \frac{r}{2 C_1 C_A C_B} \right) \right) \right)
\\
&\geq 2 C_{\mu_{\rho}}^3 \int_{B \left( b,2h^{-1} \left( \frac{r}{2 C_1 C_A C_B} \right) \right) \setminus B \left( b,h^{-1}(2C_A r) \right)} \rho(|x|)^{\s} d \mu(x)
\\
&\geq 2 C_{\mu_{\rho}}^3 \mu_{\rho}(B_{\rho}(\infty,2r) \setminus B_{\rho}(\infty,r)).
\end{align*}

Now let $z \in X$ such that $|z|=1$. Such a point exists, because the space $(X,d)$ is connected and unbounded. Then $d_{\rho}(z,\infty) \geq h(1) / C_A = 2 \rho(1) / C_A = 2 \tau_1 \geq 12 C_1 C_A^2 C_B \tau_2 = 4 C_A r$, where we used Lemma \ref{constants}.
Thus again by connectivity there exists $x \in X$ such that $d_{\rho}(x,\infty) = \frac{3}{2}r$.
Then
\begin{align*}
\mu_{\rho}(B_{\rho}(\infty,r))
&\leq \mu_{\rho}(B_{\rho}(x,4r))
\leq C_{\mu_{\rho}}^3 \mu_{\rho}(B_{\rho}(x,r/2))
\\
&\leq C_{\mu_{\rho}}^3 \mu_{\rho}(B_{\rho}(\infty,2r) \setminus B_{\rho}(\infty,r))
\leq \frac{1}{2} \mu_{\rho}(B_{\rho}(\infty,r)),
\end{align*}
which is a contradiction, because $\mu_{\rho}(B_{\rho}(\infty,r))$ is positive. Because $C_{\mu_{\rho}}$ was arbitrary, we conclude that $\mu_{\rho}$ is not doubling in $(X,d_{\rho})$ with any constant.
\end{proof}

From the previous theorem we conclude that to obtain the doubling property of $\mu_{\rho}$ we need Condition~\ref{condition-C}.
In the next example we show that Condition~\ref{condition-C} does not follow from Conditions~\ref{rhodoubling} and~\ref{upperbound}.
However if $\s \geq n$, then Conditions \ref{rhodoubling} and \ref{upperbound} imply Condition~\ref{condition-C} in $\mathbb{R}^n$, as we show in the subsequent proposition.

\begin{example}
Consider again the example function $\rho(t) = (t+2)^{\alpha} (\log(t+2))^{\beta}$.
Recall that $\rho$ satisfies \ref{rhodoubling} and \ref{upperbound} if and only if $\alpha < -1$.
In this example let $X$ be the open upper half-space $\mathbb{R}_+^n$ and let $b$ be the origin. Then
\begin{equation*}
\int_{X\setminus B(b,r)} \rho(|x|)^\s d\mu(x)
= c_n \int_r^{\infty} t^{n-1} (t+2)^{\alpha \s} (\log(t+2))^{\beta \s} dt.
\end{equation*}
If $\alpha \s > -n$, then this integral is infinite and thus $\mu_{\rho}(X) = \infty$. The same thing happens, if $\alpha \s = -n$ and $\beta \s \geq -1$.
If $\alpha \s < -n$, then we see that $\rho$ satisfies \ref{condition-C}.
If $\alpha \s = -n$ and $\beta \s < -1$, then this integral is finite, but $\rho$ does not satisfy \ref{condition-C}. 
Thus if $\s < n$, $\alpha = -n / \s < -1$ and $\beta < -1 / \s$, then the integral is finite and this function satisfies Conditions \ref{rhodoubling} and \ref{upperbound}, but it does not satisfy Condition \ref{condition-C}.
\end{example}

\begin{prop}
If $\s$ is large enough, then Condition \ref{condition-C} follows from Conditions \ref{rhodoubling} and \ref{upperbound} with constant $C_C$ depending on $C_{\mu}$, $\s$, $C_A$ and $C_B$.
\end{prop}

\begin{proof}

If $r \geq 1$, then we get from \ref{rhodoubling}, Lemma \ref{BaloghLem} and the doubling property of $\mu$
\begin{align*}
\frac{\int_{X\setminus B(b,r)}\rho(|x|)^\s d\mu(x)}{\rho(r)^\s\mu(B(b,r+1))}
&\leq \frac{\sum_{j=0}^{\infty} \int_{B(b,2^{j+1}r) \setminus B(b,2^j r)} \rho(|x|)^{\s} d\mu(x)}{\rho(r)^{\s} \mu(B(b,r))}
\\
&\leq \sum_{j=0}^{\infty} \frac{ \mu(B(b,2^{j+1} r)) C_A^{\s} \rho(2^{j+1} r)^{\s}}{\rho(r)^{\s} \mu(B(b,r))}
\\
&\leq \sum_{j=0}^{\infty} C_{\mu}^{j+1} C_A^{\s} (C_A C_B)^{\s} \left( \frac{r+1}{2^{j+1} r + 1} \right)^{\s (\eps + 1)}
\\
&\leq C_{\mu} C_A^{2 \s} C_B^{\s} \sum_{j=0}^{\infty} C_{\mu}^j 2^{-j \s (\eps + 1)}
\simeq \sum_{j=0}^{\infty} \left( \frac{C_{\mu}}{2^{\s(\eps+1)}} \right)^j
\simeq 1,
\end{align*}
if $\s > \frac{\log_2(C_{\mu})}{\eps+1}$.
This holds for any $C_A$ and $C_B$, if $\s \geq \log_2(C_{\mu})$ and in the case of $\mathbb{R}^n$, $\log_2(C_{\mu}) = n$.
If $0 < r < 1$, then we get from the above
\begin{align*}
\frac{\int_{X\setminus B(b,r)}\rho(|x|)^\s d\mu(x)}{\rho(r)^\s\mu(B(b,r+1))}
&\leq \frac{C_A^{\s} \int_X \rho(|x|)^{\s} d\mu(x)}{\rho(1)^{\s} \mu(B(b,1))}
\\
&= C_A^{\s} \frac{\int_{X \setminus B(b,1)} \rho(|x|)^{\s} d\mu(x) + \int_{B(b,1)} \rho(|x|)^{\s} d\mu(x)}{\rho(1)^{\s} \mu(B(b,1))}
\\
&\simle \frac{\rho(1)^{\s} \mu(B(b,2)) + \rho(1)^{\s} \mu(B(b,1))}{\rho(1)^{\s} \mu(B(b,1))}
\leq C_{\mu} + 1.
\qedhere
\end{align*}
\end{proof}

\subsection{Preserving the doubling property}
\label{subsect-doubling}

Throughout the rest of this paper, in addition to the finiteness of $\int_0^\infty\rho(t)\, dt$ and Conditions~\ref{rhodoubling}, \ref{upperbound}, we also make Condition~\ref{condition-C} a standing assumption.
The doubling of $\mu$ and Condition~\ref{rhodoubling} together give that
\begin{equation*}
    \int_{X\setminus B(b,r)}\rho(|x|)^\s d\mu(x)
    \geq \int_{B(b,2r+1)\setminus B(b,r)}\rho(|x|)^\s d\mu(x) \geq\frac{1}{C_A^{\s}C_{\mu}^3}\rho(r)^\s\mu(B(b,r+1))
\end{equation*}
for every $r > 0$.
Together with Condition \ref{condition-C} this gives us
\begin{equation}\label{rho and measure}
\int_{X\setminus B(b,r)}\rho(|x|)^\s d\mu(x)
\simeq \rho(r)^\s \mu(B(b,r+1))\quad\text{for }r > 0.
\end{equation}
To prove the doubling property of $\mu_{\rho}$ we start with balls at infinity.

\begin{prop}
\label{balls at infty}
For all $r>0$ we have
    $$
    0<\mu_{\rho}(B_\rho(\infty,2r))\simle \mu_{\rho}(B_\rho(\infty,r))<\infty,
    $$
    with the comparison constant depending on $C_U$, $C_{\mu}$, $\s$, $C_A$, $C_B$ and $C_C$.
\end{prop}
\begin{proof} 
    From Conditions \ref{rhodoubling} and \ref{condition-C} we get that
\begin{equation}
\label{measureofX}
\mu_{\rho}(X \setminus B(b,r))
\leq C_C \rho(r)^{\s} \mu(B(b,r+1))
\leq C_C C_A^{\s} \rho(1)^{\s} \mu(B(b,2))
\end{equation}
for every $0 < r \leq 1$.
Thus by taking the limit $r \to 0$, we get that $\mu_{\rho}(X) < \infty$ and this shows the final inequality.
Moreover, since $\mu(B_\rho(\infty,r))>0$, see Lemma \ref{small balls at infty}, and $\rho>0$, we have that $\mu_{\rho}(B_\rho(\infty,r))>0$. 
Let $r_0=\frac{\tau_1}{2C_A}$.
If $0<r\le r_0$, then the claim follows by Lemmas \ref{small balls at infty} and \ref{h-1doubling} together with \eqref{rho and measure}.
If $r\ge r_0$, then similarly we get
    \begin{equation*}
    \mu_{\rho}(B_\rho(\infty,r))\ge \mu_{\rho}(B_\rho(\infty,r_0))\simge
    \mu_{\rho}(X)\ge\mu_{\rho}(B_\rho(\infty,2r)).
    \qedhere
    \end{equation*}
\end{proof}

Before we prove the doubling property for balls centered at other points than $\infty$, we need the following lemma which compares balls, centered at points in $X$, with respect to the metric $d_\rho$, to balls with respect to $d$. In particular, this lemma indicates that the two metrics $d$ and $d_\rho$ are locally quasisymmetrically equivalent, but we will not need this property in this paper.

\begin{lem}
\label{lem-13}
Let $c_0 := (2 C_1 C_A^2 C_B)^{-1}$, where $C_1 = 3C_U\, C_A^{4+\log_2(C_U)}\, C_B$ is the constant from Lemma~\ref{constants}.
If $x \in X$ and $0 < r \leq 2 c_0 d_{\rho}(x,\infty)$, then
\begin{equation*}
B \left( x,\frac{a_1 r}{\rho(|x|)} \right)
\subset B_{\rho}(x,r)
\subset B \left( x,\frac{a_2 r}{\rho(|x|)} \right),
\end{equation*}
where $a_1=1/C_2$, $a_2=C_A^3C_B$ and $C_2 = C_U C_A^{3+\log_2(C_U)}$ is the constant from Lemma \ref{lem-3}.
Also $\frac{1}{C_A}\rho(|x|) \leq \rho(|y|) \leq C_A \rho(|x|)$ for every $y \in B_{\rho}(x,r)$.
\end{lem}

\begin{proof}
Recall that $C_A C_B \geq 1$ and $C_A > 2$.
Assume that $y \in B_{\rho}(x,r)$. If $|y| \geq 2|x|+1$, then by Lemma \ref{lem-4} and Lemma \ref{constants}
\begin{equation*}
2c_0 d_{\rho}(x,\infty)
\geq r
> d_{\rho}(x,y)
\geq \frac{1}{C_A} \rho(|x|)(|x|+1)
\geq \frac{1}{C_A C_1} d_{\rho}(x,\infty).
\end{equation*}
By the definition of $c_0$, this is a contradiction.
If $|x| \geq 2|y|+1$, then by Lemma \ref{lem-4} and Lemma \ref{constants} together with Remark \ref{remark-1}(ii) we have
\begin{align*}
d_{\rho}(x,\infty)
&\leq C_1 \rho(|x|)(|x|+1)
\leq C_1 C_A C_B \rho(|y|)(|y|+1)
\leq C_1 C_A^2 C_B d_{\rho}(x,y)
\\
&< C_1 C_A^2 C_B r
\leq 2 C_1 C_A^2 C_B c_0 d_{\rho}(x,\infty).
\end{align*}
Again we have a contradiction by the definition of $c_0$.
Thus $y \in B_{\rho}(x,r)$ implies that $|y| \leq 2|x|+1$ and $|x| \leq 2|y|+1$. Then Condition~\ref{rhodoubling} gives us $\rho(|y|) \simeq \rho(|x|)$.

Let us now prove the second inclusion. Assume that $y \in B_{\rho}(x,r)$. Then $\frac{1}{2} (|x|+~1) \leq |y|+1 \leq 2(|x|+1)$.
Thus by Lemma \ref{lem-3} we have
\begin{equation*}
d(x,y)
\leq \frac{C_A^3 C_B d_{\rho}(x,y)}{\rho(|x|)}
< \frac{C_A^3 C_B r}{\rho(|x|)},
\end{equation*}
so the second inclusion holds.

To prove the first inclusion suppose that $y \in B \left( x , \frac{a_1 r}{\rho(|x|)} \right) $. 
Then because $\frac{1}{C_2} \leq \frac{1}{4 c_0 C_1}$,
\begin{equation*}
d(x,y)
< \frac{r}{4c_0 C_1 \rho(|x|)}
\leq \frac{2c_0 d_{\rho}(x,\infty)}{4 c_0 C_1 \rho(|x|)}
\leq \frac{1}{2}(|x|+1),
\end{equation*}
where we used Lemma \ref{constants}.
It thus follows from the triangle inequality that
\begin{equation*}
\frac{1}{2}(|x|+1)
\leq |y|+1
\leq 2(|x|+1).
\end{equation*}
Thus by Lemma \ref{lem-3},
\begin{equation*}
d_{\rho}(x,y)
\leq C_2 \rho(|x|) d(x,y)
< C_2 \rho(|x|) a_1 r / \rho(|x|)
= r,
\end{equation*}
which completes the proof.
\end{proof}

Now we are ready to prove the main result of this section, which verifies Theorem~\ref{maintheorem}(b).
\begin{thm}
\label{thm-doubling}
Let $(X,d,\mu)$ be a $C_U$-uniform space with a doubling measure $\mu$ and assume that $\rho$ satisfies Conditions \ref{rhodoubling}, \ref{upperbound} and \ref{condition-C}.
Then $\mu_{\rho}$ is doubling in the space $(X,d_{\rho})$ i.e. there exists a constant $C_{\mu_{\rho}} > 1$ that depends on $C_U$, $C_{\mu}$, $\s$, $C_A$, $C_B$ and $C_C$ such that
\begin{equation*}
0
< \mu_{\rho} (B_{\rho}(x,2r))
\leq C_{\mu_{\rho}} \mu_{\rho} (B_{\rho}(x,r))
< \infty
\end{equation*}
for every $x \in X$ and $r > 0$.
\end{thm}

\begin{proof}
We get the final inequality the same way as in \eqref{measureofX}.
For the first inequality we get from Lemma \ref{lem-13} that
\begin{align*}
\mu_{\rho}(B_{\rho} (x,r))
&\geq \mu_{\rho}(B_{\rho} (x,\min\{r,2c_0d_{\rho}(x,\infty)\}))
\\
&\geq \mu_{\rho} \left( B \left( x,\frac{a_1 \min\{r,2c_0d_{\rho}(x,\infty)\}}{\rho(|x|)} \right) \right).
\end{align*}
This is positive, because $\rho$ is positive and the $\mu$-measure of the ball is positive, because $a_1 \min\{r,2c_0d_{\rho}(x,\infty)\} / \rho(|x|) > 0$.

The rest of the proof is divided into the following three cases: 
\begin{enumerate}
\item[(1)] $0 < r \leq c_0 d_{\rho}(x,\infty)$,
\item[(2)] $r \geq 2 d_{\rho}(x,\infty)$, 
\item[(3)] $c_0 d_{\rho}(x,\infty) \leq r \leq 2 d_{\rho}(x,\infty)$.
\end{enumerate}
Here $c_0$ is the constant in Lemma \ref{lem-13}.

In Case (1) we get from Lemma \ref{lem-13} and the doubling property of $\mu$ that
\begin{align*}
\mu_{\rho}(B_{\rho}(x,2r))
&\simeq \rho(|x|)^{\s} \mu(B_{\rho}(x,2r))
\leq \rho(|x|)^{\s} \mu \left( B \left( x, \frac{2 a_2 r}{\rho(|x|)} \right) \right)
\\
&\simeq \rho(|x|)^{\s} \mu \left( B \left( x, \frac{a_1 r}{\rho(|x|)} \right) \right)
\leq \rho(|x|)^{\s} \mu(B_{\rho}(x,r))
\simeq \mu_{\rho}(B_{\rho}(x,r)).
\end{align*}
This completes the proof of Case (1).

For Case (2) we have
\begin{equation*}
\mu_{\rho}(B_{\rho}(x,2r))
\leq \mu_{\rho}(B_{\rho}(\infty,3r))
\simle \mu_{\rho}(B_{\rho}(\infty,r/2))
\leq \mu_{\rho}(B_{\rho}(x,r)),
\end{equation*}
where we used Proposition \ref{balls at infty}.

Now we move on to Case (3). We shall show that
\begin{equation}
\label{thirdcase}
\mu_{\rho}(B_{\rho}(x,r))
\simeq \rho(|x|)^{\s} \mu(B(x,|x|+1)),
\end{equation}
if $c_0 d_{\rho}(x,\infty) \leq r \leq 4 d_{\rho}(x,\infty)$. Since the right hand side does not depend on $r$, this will complete the proof. Here the comparison constants depend on $C_U$, $C_{\mu}$, $\s$, $C_A$, $C_B$ and $C_C$.

We start by proving the "$\simle$" part of \eqref{thirdcase}. Thus let us assume that $y \in B_{\rho}(x,r)$.
The first step is to show that there exists a constant $M \geq 2$ that depends only on $C_U$, $C_A$ and $C_B$, such that $|x|+1 \leq M(|y|+1)$. If $|x|+1 \leq 2(|y|+1)$, the claim holds for any $M \geq 2$.
If $|x|+1 \geq 2(|y|+1)$ i.e. $|x| \geq 2|y|+1 \geq |y|$, then we get from Lemmas \ref{BaloghLem} and \ref{lem-4} together with Lemma \ref{constants} that
\begin{align*}
\rho(|y|)(|y|+1)
&\leq C_A d_{\rho}(x,y)
< 4 C_A d_{\rho}(x,\infty)
\leq 4 C_A C_1 \rho(|x|)(|x|+1)
\\
&\leq 4 C_A^2 C_B C_1 \left( \frac{|y|+1}{|x|+1} \right)^{\eps} (|y|+1) \rho(|y|),
\end{align*}
where we also used the fact that $d_{\rho}(x,y) < r \leq 4d_{\rho}(x,\infty)$. Therefore by setting $M := (4 C_A^2 C_B C_1)^{C_A C_B} $ we have $|x|+1 \leq M(|y|+1)$.

The next step is to deal separately with the cases $|x|+1 \leq M$ and $|x|+1 > M$. We start with $|x|+1 \leq M$. In this case we get from Condition \ref{rhodoubling} that $\rho(1) \simeq \rho(|x|)$ and therefore
\begin{align*}
\mu_{\rho}(B_{\rho}(x,r))
&\leq \mu_{\rho}(X)
\simle \rho(1)^{\s} \mu(B(b,1))
\simle \rho(|x|)^{\s} \mu(B(b,|x|+1))
\\
&\leq \rho(|x|)^{\s} \mu(B(x,2(|x|+1)))
\leq \rho(|x|)^{\s} C_{\mu} \mu(B(x,|x|+1)),
\end{align*}
where we used \eqref{measureofX} and the doubling property of $\mu$.

Now let $|x|+1 > M$.
Let us denote $r_x := \frac{1}{M}(|x|+1) - 1 > 0$. Then $|y| \geq r_x$ for every $y \in B_{\rho}(x,r)$ and $\rho(|x|) \simeq \rho(r_x)$ by Condition \ref{rhodoubling}. Thus by Condition \ref{condition-C} and the doubling property of $\mu$
\begin{align*}
\mu_{\rho}(B_{\rho}(x,r))
&= \int_{B_{\rho}(x,r)} \rho(|y|)^{\s} d \mu(y)
\leq \int_{X \setminus B(b,r_x)} \rho(|y|)^{\s} d \mu(y)
\\
&\leq C_C \rho(r_x)^{\s} \mu(B(b,r_x+1))
= C_C \rho(r_x)^{\s} \mu \left( B \left( b,\frac{1}{M}(|x|+1) \right) \right)
\\
&\simeq \rho(|x|)^{\s} \mu(B(x,|x|+1))
\end{align*}
with the comparison constants depending on $M$, $C_{\mu}$, $\s$, $C_A$ and $C_C$.
Thus we have shown the "$\simle$" part of \eqref{thirdcase}.

Now let us prove the "$\simge$" part of \eqref{thirdcase}.
With $C_2$ the constant in Lemma \ref{lem-3} we show that
\begin{equation*}
B \Big( x,\tfrac{c_0}{C_A C_2}(|x|+1) \Big)
\subset B_{\rho}(x,r).
\end{equation*}
So assume that $y \in B \left( x,\tfrac{c_0}{C_A C_2}(|x|+1) \right)$.
Then by the triangle inequality
\begin{equation*}
\frac{1}{2}(|x|+1)
\leq |y|+1
\leq 2(|x|+1),
\end{equation*}
because $\frac{c_0}{C_A C_2} \leq \frac{1}{2}$.
Thus by Condition \ref{rhodoubling} we have $\rho(|y|) \simeq \rho(|x|)$.
We get from Lemmas \ref{lem-3} and \ref{lem-4} that
\begin{equation*}
d_{\rho}(x,y)
\leq C_2 \rho(|x|) d(x,y)
< \frac{c_0}{C_A} \rho(|x|) (|x|+1)
\leq c_0 d_{\rho}(x,\infty)
\leq r.
\end{equation*}
Therefore $B \left( x,\frac{c_0}{C_A C_2}(|x|+1) \right) \subset B_{\rho}(x,r)$ and thus
\begin{align*}
\mu_{\rho}(B_{\rho}(x,r))
&= \int_{B_{\rho}(x,r)} \rho(|y|)^{\s} d \mu(y)
\geq \int_{B \left( x,\frac{c_0}{C_A C_2}(|x|+1) \right) } \rho(|y|)^{\s} d \mu(y)
\\
&\simeq \rho(|x|)^{\s} \mu \left( B \left( x,\tfrac{c_0}{C_A C_2}(|x|+1) \right) \right)
\simeq \rho(|x|)^{\s} \mu(B(x,|x|+1)),
\end{align*}
where we used the doubling property of $\mu$. This completes the proof of \eqref{thirdcase}.
\end{proof}

\section{On preservation of Poincar\'e inequalities}
\label{sect-poincare}

One of the principal foci of this paper is the geometric-analytic notion of Poincar\'e inequality. Given the key role this inequality plays in much of non-smooth analysis, we are interested in establishing that sphericalization of a metric measure space supporting a Poincar\'e inequality results in a space that also supports a Poincar\'e inequality. The upper gradient structure referred to in the Poincar\'e inequality is classically played by the magnitude of the weak derivative of a Sobolev function in Euclidean spaces. The advantage of the notion of upper gradients is that there is no need of a smooth structure, but the upper gradient does depend on the underlying metric on the space. In this section we will also describe how the upper gradients transform when the metric is transformed from $d$ to $d_\rho$.

We now define the notion of supporting a $p$-Poincar\'e inequality.
In what follows, if $(Z,d_Z)$ is a metric space and $u$ is an extended real-valued function on $Z$, we say that a non-negative Borel measurable function $g$ on $Z$ is an \emph{upper gradient} of $u$, if for each non-constant compact rectifiable curve $\gamma:[s,t]\to Z$ we have that
\[
|u(\gamma(t))-u(\gamma(s))|\le \int_\gamma g(\cdot)ds_{\textrm{\tiny Z}},
\]
where, whenever at least one of $|u(\gamma(s))|$ and $|u(\gamma(t))|$ is infinite, the above is interpreted also to mean that $\int_\gamma g(\cdot)ds_\textrm{\tiny Z}=\infty$.
The notation $ds_{\textrm{\tiny Z}}$ in the path integral means that we are calculating the path integral in $Z$ with respect to the metric $d_Z$.
The notion of upper gradients is due to Heinonen and Koskela~\cite[Section~2.9]{HK98}.
If $B=B_Z(z,r):=\{x \in Z : d_Z(z,x)<r\}$ is a ball, we will use the notation $\lambda B=B_Z(z,\lambda r)$ and the number $u_{B}$ is defined by
\begin{equation*}
u_{B}
:= \vint_{B} u(x)\, d\mu_Z(x)
:= \tfrac{1}{\mu_Z(B)}\, \int_{B}u(x)d\mu_Z(x).
\end{equation*}

\begin{deff}
Let $1\le p<\infty$.
Let $(Z,d_Z,\mu_Z)$ be a metric measure space with $\mu_Z$ a Borel measure supported on $Z$. We say that this metric measure space $Z$ \emph{supports a $p$-Poincar\'e inequality}, if there exist constants $C_P>0$ and $\lambda\ge 1$ such that for every ball $B=B_Z(z,r)$,
\begin{equation}
\label{eq-poincare}
\vint_{B}|u(x)-u_{B}|\, d\mu_Z(x) \le C_P\, r\, \left(\vint_{\lambda B}g(x)^p\, d\mu_Z(x) \right)^{1/p},
\end{equation}
whenever $z\in Z$, $r>0$, $u\in L^1(\lambda B)$ and $g$ is an upper gradient of $u$ in the ball $\lambda B$.
By a truncation argument we can see that if the Poincar\'e inequality holds for bounded measurable functions, then it holds for all measurable functions and hence the Poincar\'e inequality implies that any measurable function that has an upper gradient in $L^p(\lambda B)$ is necessary also in $L^1(B)$. 
\end{deff}

The inequality \eqref{eq-poincare} is equivalent to the following, if we allow the constant $C_P$ to change: there is a constant $C_P>0$ such that
\[
\inf_{c\in\mathbb{R}}\vint_{B}|u(x)-c|\, d\mu_Z(x) \le C_P \, r\, \left(\vint_{\lambda B}g(x)^p\, d\mu_Z(x) \right)^{1/p}.
\]

Throughout this section, we continue to study $(X,d_\rho, \mu_\rho)$ with the assumption that $(X,d,\mu)$ is doubling and uniform and supports a $p$-Poincar\'e inequality with constants $C_P$ and $\lambda$. We also assume that $\rho$ is lower semicontinuous and satisfies Conditions~\ref{rhodoubling}, \ref{upperbound} and~\ref{condition-C}.
The goal of this section is to use the arguments developed in~\cite{BSh-Uniform} together with Theorem~\ref{thm-uniformity}, Lemma~\ref{lem-13} and Theorem~\ref{thm-doubling} above to conclude that $(X,d_\rho,\mu_\rho)$ supports a $p$-Poincar\'e inequality as well.

\subsection{Path integrals under lower semicontinuity of \texorpdfstring{$\rho$}{}}

In this subsection we show some properties of path integrals.
Our assumptions on $\rho$ allow us to use Proposition~\ref{homeo} to infer that the topology induced on $X$ by the metric $d$ and the topology induced on $X$ by the metric $d_\rho$ coincide.

\begin{lem}\label{lem:lsc}
If $\gamma$ is a rectifiable curve,
then
\[
\ell_\rho(\gamma)=\int_\gamma\rho(|\cdot|)\, ds.
\]
\end{lem}
\begin{proof}
As $\rho\colon (0,\infty)\to (0,\infty)$ is lower semicontinuous, there exists a monotone increasing sequence of positive continuous functions $(\rho_k)_k$ that converges pointwise to $\rho$ (see for example \cite[Exercise~22, p.~60]{Rudin}).
Since $\rho_k$ are positive and continuous and $(X,d)$ is quasiconvex we can apply \cite[Proposition A.7]{BHK} to obtain
  \[
  \ell_{\rho_k}(\gamma)=\int_\gamma \rho_k(|\cdot|) ds\longrightarrow_{k \to \infty} \int_\gamma \rho(|\cdot|) ds,
  \]
  where the convergence follows from the Monotone Convergence Theorem. As $\rho_k\le \rho$, we have $\ell_{\rho_k}(\gamma)\le \ell_\rho(\gamma)$ for every $k$. Consequently, we have
  \[
  \int_\gamma \rho(|\cdot|) \,ds\le \ell_\rho(\gamma),
  \]
  which completes the proof as by \eqref{rholength} the opposite inequality holds as well.
\end{proof}

In the following we will use the notation $ds_\rho$ in path integrals for specifying that we are calculating the path integral in $X$ with respect to the metric $d_\rho$, and we continue to use $ds$ when we are calculating the path integral in $X$ with respect to the metric $d$.

\begin{lem}\label{lem:lsc-rho}
If $\gamma$ is a rectifiable curve in $(X,d)$ and $g$ is a non-negative Borel measurable function on $(X,d)$, then 
\[
\int_\gamma g(\cdot)\,ds_\rho=\int_\gamma \rho(|\cdot|) g(\cdot)\,ds.
\]
\end{lem}

\begin{proof}
Let $\gamma\colon[0,\ell_d(\gamma)]\to X$ be the arc-length parametrization of $\gamma$ with respect to the metric $d$. 
Let $s_\rho$ denote its associated length function with respect to the metric $d_\rho$, i.e. $s_\rho(t)=\ell_\rho(\gamma|_{[0,t]})=\int_0^t\rho(|\gamma(\tau)|)\,d\tau$, where the last equality follows from Lemma~\ref{lem:lsc}. 
Since by Proposition~\ref{rectifiability} the curve  $\gamma$ is rectifiable in $(X,d_{\rho})$, the function $\rho(|\gamma(\cdot)|)$ is integrable on $[0,\ell_{d}(\gamma)]$ and thus $s_{\rho}$ is absolutely continuous.
This implies that the  parametrization of $\gamma$ is absolutely continuous with respect to $d_\rho$. 
Since $s_\rho'(t)=\rho(|\gamma(t)|)$ almost everywhere, we obtain by~\eqref{eqn:pathintegral}
    \begin{equation*}
        \int_\gamma g(\cdot)\,ds_\rho=\int_0^{\ell_d(\gamma)}g(\gamma(\tau))s_\rho'(\tau)\,d\tau
        =\int_0^{\ell_d(\gamma)}g(\gamma(\tau))\rho(|\gamma(\tau)|)\,d\tau=\int_\gamma\rho(|\cdot|)g(\cdot)\,ds.
        \qedhere
    \end{equation*}
\end{proof}

The next proposition is an immediate consequence of Lemma~\ref{lem:lsc-rho} and Definition~\ref{murhodef}.
The proposition is of independent interest as it demonstrates how weak upper gradients transform under conformal change in the metric, if the metric is given by a density function $\rho$ that is lower semicontinuous. Thus it is applicable even in cases where $\rho$ may not satisfy any of Conditions~\ref{condA}--\ref{condC}.

\begin{prop}\label{prop:ug-transform}
Let $u:X\to\mathbb{R}$ and $g:X\to[0,\infty)$ be two measurable functions with $g$ a Borel measurable function. Then the following hold true:
\begin{enumerate}
\item $g\in L^p(X,\mu)$ if and only if $\rho(|\cdot|)^{-\s/p} g(\cdot)\in L^p(X,\mu_\rho)$.
\item $g$ is an upper gradient of $u$ in $(X,d_\rho)$ if and only if $\rho(|\cdot|)\, g(\cdot)$ is an upper gradient of $u$ in $(X,d)$.
\end{enumerate}
\end{prop}

It follows from the above proposition that if we want to use the tools of sphericalization to study $p$-minimizers,  the choice $\sigma=p$ is useful as then the $p$-minimizers are preserved in $X$. However, when we discuss the preservation of $p$-Poincar\'e inequality, the result holds for all $p\ge 1$ regardless of the choice of $\sigma$ as long as Condition~\ref{condC} is satisfied.
Notice that as Conditions~\ref{condA} and~\ref{condB} imply that $\rho$ is quasidecreasing, it is easy to see that if Condition~\ref{condC} holds with some $\sigma$, then it holds also with all larger values of~$\sigma$.

\subsection{Preserving the \texorpdfstring{$p$}{}-Poincaré inequality}

\begin{lem}\label{lem:GoIn}
Let $z\in X$ and $0<r \leq 2\diam_\rho(X)$. 
Then there exists $z_0\in B_{\rho}(z,r)$ such that $d_{X,\rho}(z_0)\ge r/(16C_{U,\rho})$.
\end{lem}

\begin{proof}
Since $r\le 2 \diam_\rho(X)$, there is $w\in B_\rho(z,r)\setminus B_\rho(z,\frac18r)$. Let $\gamma$ be a $C_{U,\rho}$-uniform curve in $(X,d_\rho)$ with end points $w$ and $z$. Thus $\ell_\rho(\gamma)\ge \frac18r$ and there exists $z_0$ in the trajectory of the curve such that $\frac{1}{16}r=\ell_\rho(\gamma_{zz_0})\le\ell_\rho(\gamma_{z_0w})$. Now $z_0\in B_\rho(z,r)$ and by the twisted cone condition
\begin{equation*}
    d_{X,\rho}(z_0)\ge \frac{1}{C_{U,\rho}}\ell_\rho(\gamma_{zz_0})=\frac{1}{16C_{U,\rho}}r.\qedhere
\end{equation*}
\end{proof} 

We recall the following key lemma from~\cite{BSh-Uniform}.

\begin{lem}[{\cite[Lemma~4.3]{BSh-Uniform}}]
\label{lem:BomanChain}
Let $\tau\ge1$, $r>0$, $z_0\in X$ and $0<\rho_0\le \frac{d_{X,\rho}(z_0)}{16\tau C_{U,\rho}}$, where $C_{U,\rho}$ is the uniformity constant of $(X,d_\rho)$. Then every $x\in B_\rho(z_0,2r)$ can be connected to the ball $B_{0,0}:=B_\rho(z_0,\rho_0)$ by a chain of balls $\{B_{i,j}\, :\, i\in\mathbb{N}\cup\{0\}\text{ and }j=0,\cdots, m_i\}$ satisfying the following conditions:
\begin{enumerate}
\item for $i=0,1,\cdots$, the radius $rad(B_{i,j})=2^{-i}\rho_0$ and the ball $B_{i,j}$ has center $x_{i,j}$ such that $d_{X,\rho}(x_{i,j})\ge2^{2-i}\tau \rho_0$ and $d_\rho(x_{i,j},x) \le 2^{-i}C_{U,\rho}\, d_\rho(x,z_0)<2^{1-i}C_{U,\rho}r$ for $j=0,\cdots, m_i$,
\item for each $i=0,1,\cdots$, we have that $m_i\le 2C_{U,\rho}r/\rho_0$,
\item there is some positive integer $i_x$ such that whenever $i\ge i_x$ we have that $m_i=0$ and $x_{i,0}=x$,
\item with the lexicographic ordering of these balls, if $B_{i,j}^*$ is a neighbor of $B_{i,j}$ in this ordering, then $B_{i,j}\cap B_{i,j}^*$ is nonempty.
\end{enumerate}
The radii referred to above are with respect to the metric $d_\rho$.
\end{lem}

In the rest of this section we choose $\tau = \frac{5 \lambda a_2}{8 a_1 c_0}$, where $c_0$, $a_1$ and $a_2$ are as in Lemma~\ref{lem-13}. The reason for this choice can be seen in the proof of the next lemma.
In the rest of this section we call the constants $C_U$, $C_{\mu}$, $\s$, $p$, $\lambda$, $C_P$, $C_A$, $C_B$ and $C_C$ the structural data.

\begin{lem}\label{lem:PI-subWhitney}
Let $B_{i,j}$ be the chain of balls in Lemma \ref{lem:BomanChain}.
Let $u$ be a bounded measurable function and $g$ an upper gradient of $u$ with respect to the metric $d_\rho$.
Then for $i=0,1,\cdots$ and $j=0,\cdots, m_j$, we have that
\[
\vint_{B_{i,j}}|u(y)-u_{B_{i,j}^*}|\, d\mu_\rho(y)
\simle 2^{-i}\rho_0\, \left(\vint_{\frac{5\lambda a_2}{a_1}\, B_{i,j}}g(y)^p\, d\mu_\rho(y) \right)^{1/p},
\]
where $B_{i,j}^*$ is a lexicographic neighbour of $B_{i,j}$ and $a_1,a_2$ are as in Lemma~\ref{lem-13}.
The comparison constant depends only on the structural data. 
\end{lem}

\begin{proof}
With our choice of $\tau$ we get from Lemma \ref{lem:BomanChain}(a) that $0 < \frac{5 \lambda a_2}{a_1} 2^{-i} \rho_0 \leq 2 c_0 d_{\rho}(x_{i,j},\infty)$.
Therefore we have from Lemma \ref{lem-13} that $\rho(|y|) \simeq \rho(|x_{i,j}|)$ for every $y \in \frac{5 \lambda a_2}{a_1} B_{i,j}$.
By setting $\widehat{B_{i,j}}:=B(x_{i,j},2^{-i}\rho_0a_1/\rho(|x_{i,j}|))$ we also get that
\begin{align*}
5\widehat{B_{i,j}}
&\subset 5B_{i,j}
\subset \frac{5a_2}{a_1}\widehat{B_{i,j}} \quad \text{ and }
\\
\frac{5\lambda a_2}{a_1} \widehat{B_{i,j}}
&\subset \frac{5 \lambda a_2}{a_1} B_{i,j}
\subset \frac{5\lambda a_2^2}{a_1^2} \widehat{B_{i,j}}.
\end{align*}
These properties are used repeatedly below.

Note that $B_{i,j}^* \subset 5B_{i,j}$. It follows from the doubling property of $\mu_\rho$ that
\begin{align*}
\vint_{B_{i,j}}|u(y)-u_{B_{i,j}^*}|\, d\mu_\rho(y)
&\lesssim \vint_{5B_{i,j}}\vint_{5B_{i,j}}|u(y)-u(w)|\, d\mu_\rho(y)\, d\mu_\rho(w)
\\
&\le 2\, \vint_{5B_{i,j}}|u(y)-c_{i,j}|\, d\mu_\rho(y),
\end{align*}
where
\[
c_{i,j}=
\frac{1}{\mu(\frac{5a_2}{a_1} \widehat{B_{i,j}})}\, \int_{\frac{5a_2}{a_1} \widehat{B_{i,j}}}u(y)d\mu(y).
\]
As noted in the beginning of this proof
\[
\mu_\rho(5B_{i,j})\simeq \rho(|x_{i,j}|)^{\s} \mu(5B_{i,j}),
\]
and so 
\[
\vint_{5B_{i,j}}|u(y)-c_{i,j}|\, d\mu_\rho(y) \simeq \vint_{5B_{i,j}}|u(y)-c_{i,j}|\, d\mu(y) \lesssim \vint_{\tfrac{5a_2}{a_1}\widehat{B_{i,j}}}|u(y)-c_{i,j}|\, d\mu(y).
\]
We now apply the $p$-Poincar\'e inequality on the ball $\tfrac{5a_2}{a_1}\widehat{B_{i,j}}$ to see that
\[
\vint_{\tfrac{5a_2}{a_1}\widehat{B_{i,j}}}|u(y)-c_{i,j}|\, d\mu(y)
\leq C_P \frac{5 a_2 2^{-i} \rho_0}{\rho(|x_{i,j}|)} \left(\vint_{\tfrac{5\lambda a_2}{a_1}\widehat{B_{i,j}}} (\rho(|y|) g(y))^p\, d\mu(y) \right)^{1/p}.
\]
We used Proposition~\ref{prop:ug-transform}(b) here.
Therefore
\begin{align*}
\vint_{B_{i,j}}|u(y)-u_{B_{i,j}^*}|\, &d\mu_\rho(y)
\simle \frac{2^{-i} \rho_0}{\rho(|x_{i,j}|)} \left( \frac{1}{\mu(\widehat{B_{i,j}})} \int_{\tfrac{5\lambda a_2}{a_1} B_{i,j}} (\rho(|y|)g(y))^p\, d\mu(y) \right)^{1/p}
\\
&\simeq 2^{-i} \rho_0 \left( \frac{1}{\rho(|x_{i,j}|)^{\s} \mu(\widehat{B_{i,j}})} \int_{\tfrac{5\lambda a_2}{a_1} B_{i,j}} g(y)^p\, \rho(|y|)^{\s} d\mu(y) \right)^{1/p}
\\
&\simeq 2^{-i} \rho_0 \left( \vint_{\tfrac{5\lambda a_2}{a_1} B_{i,j}} g(y)^p\, d\mu_{\rho}(y) \right)^{1/p}.
\qedhere
\end{align*}
\end{proof}

Now we prove the main result of this section, which also verifies Theorem \ref{maintheorem}(c).

\begin{thm}
Let $(X,d,\mu)$ be a $C_U$-uniform space with a doubling measure $\mu$ such that the space supports a $p$-Poincaré inequality.
Suppose that $\rho$ is lower semicontinuous and satisfies Conditions \ref{condA}, \ref{condB} and \ref{condC}.
Then the space $(X,d_\rho,\mu_\rho)$ supports a $p$-Poincar\'e inequality with comparison constant that depends on the structural data, and dilation constant that depends only on $C_U$, $C_A$ and $C_B$.
\end{thm}

\begin{proof}
Let $z \in X$ and $r>0$. If $r>2\diam_\rho(X)$, then the ball $B_\rho(z,r)=X=B_\rho(z,2\diam_\rho(X))$ and thus we can assume that $0<r\le 2\diam_\rho(X)$. By the discussion in the definition of the Poincar\'e inequality, it suffices to prove the inequality for bounded functions. Thus let $u$ be a bounded measurable function and let $g$ be an upper gradient of $u$ with respect to the metric $d_\rho$.

By Lemma~\ref{lem:GoIn} we can choose $z_0\in B_\rho(z,r)$ such that $d_{X,\rho}(z_0)\ge \frac{1}{16C_{U,\rho}}r$. Let $\rho_0=\frac{1}{16C_{U,\rho}}\frac{1}{16\tau C_{U,\rho}}r$, where $\tau=\frac{5\lambda a_2}{8a_1c_0}$ as discussed above. Then $\rho_0$ is suitable for Lemma~\ref{lem:BomanChain} and as there set $B_{0,0}:=B_\rho(z_0,\rho_0)$. Let $\{B_{i,j}\}_{(i,j)}$ be the chain of balls promised by the lemma connecting $x\in B_{\rho}(z_0,2r)$ with $B_{0,0}$.

We only focus on points $x \in B_\rho(z,r)\subset B_\rho(z_0,2r)$ that are $\mu_\rho$--Lebesgue points of $u$. Then, denoting by $B^*_{i,j}$ the ball that is next to the ball $B_{i,j}$ in the lexicographic ordering referred to in Lemma~\ref{lem:BomanChain}, by Lemma~\ref{lem:PI-subWhitney}, we have
\begin{align*}
|u(x)-u_{B_{0,0}}| &\le \sum_{i,j}|u_{B_{i,j}}-u_{B_{i,j}^*}|\\
  &\leq C \rho_0\, \sum_{i=0}^\infty 2^{-i}\, \sum_{j=0}^{m_i} \left(\vint_{\frac{5\lambda a_2}{a_1} B_{i,j}}g(y)^p\, d\mu_\rho(y) \right)^{1/p}.
\end{align*}
Here the constant $C$ depends on the structural data.
We use Cavalieri's principle to estimate $\int_{B_{\rho}(z,r)}|u(x)-u_{B_{0,0}}|\, d\mu_\rho(x)$.
To do so, we fix a choice of $\eps\in(0,1)$, and for $t>0$ we consider
\[
E_t:=\{x\in B_\rho(z,r) : |u(x)-u_{B_{0,0}}|>t\text{ and }x\text{ is a Lebesgue point of }u\}.
\]
For $x\in E_t$, from the above inequality we see that
\[
t =\sum_{i=0}^\infty M c_\eps\, 2^{-i\eps}\, M^{-1} t
< \sum_{i=0}^\infty C\, \rho_0\, 2^{-i}\, \sum_{j=0}^{m_i} \left(\vint_{\frac{5\lambda a_2}{a_1} B_{i,j}}g(y)^p\, d\mu_\rho(y) \right)^{1/p}.
\]
In the above, $c_\eps=(\sum_{i=0}^\infty 2^{-i \eps})^{-1}$ and $M=1+2C_{U,\rho}r/\rho_0$ is an upper bound for $m_i+1$.
It follows that there is some ball $B_x=B_{i_x,j_x}$ in the chain $\{B_{i,j}\}_{i,j}$ such that
\[
M^{-1}\, t< C c_{\eps}^{-1} 2^{-i_x(1-\eps)}\rho_0\, \left(\vint_{\frac{5\lambda a_2}{a_1} B_x}g(y)^p\, d\mu_\rho(y) \right)^{1/p},
\]
and so, using the fact that $M\rho_0 \le 3 C_{U,\rho}r$, we get
\begin{equation}\label{eq:est1}
2^{i_x(1-\eps)}
\leq 3C\, C_{U,\rho} c_{\eps}^{-1} \frac{r}{t} 
\, \left(\vint_{\frac{5\lambda a_2}{a_1}\, B_{x}}g(y)^p\, d\mu_\rho(y) \right)^\frac1p.
\end{equation}

Since $\mu_\rho$ is doubling by Theorem \ref{thm-doubling}, there is a constant $Q>0$, depending only on the doubling constant $C_{\mu_{\rho}}$, such that
\[
\left(\frac{2^{-i}\rho_0}{\rho_0}\right)^Q\lesssim \frac{\mu_\rho(B_{i,j})}{\mu_\rho(B_{0,0})},
\]
where the comparison constant depends on $C_{U,\rho},\,\tau$ and $C_{\mu_{\rho}}$, see for example \cite[Lemma 3.3]{BBbook}.
Now setting $\alpha:=1-(1-\eps)p/Q$ and choosing $\eps\in(0,1)$ close to one so that $\alpha>0$, it follows from~\eqref{eq:est1} that
\begin{equation}\label{eq:est2}
\mu_\rho(B_{x})^\alpha\lesssim \frac{1}{\mu_{\rho}(B_{0,0})^{1-\alpha}}\, \frac{r^p}{t^p}\, 
  \int_{\frac{5\lambda a_2}{a_1}\, B_x}g(y)^p\, d\mu_\rho(y).
\end{equation}

By Lemma \ref{lem:BomanChain}(a) we have $d_{\rho}(x,x_{i_x,j_x})< C_{U,\rho} 2^{1-i_x}r =2\cdot16^2C_{U,\rho}^3\tau2^{-i_x}\rho_0$. Thus $x\in \tilde C B_x$, where $\tilde C=2\cdot16^2C_{U,\rho}^3\tau$.  
Therefore the collection $\{\tilde CB_x\, :\, x\in E_t\}$ is a cover of $E_t$; thus by the 5-covering lemma we have a countable pairwise disjoint subcollection $\mathcal{F}\subset \{\tilde CB_x\, :\, x\in E_t\}$ such that
\[
E_t\subset\bigcup_{B\in\mathcal{F}}5B.
\]
It follows from~\eqref{eq:est2} above and the doubling property of $\mu_\rho$, that
\begin{align*}
\mu_\rho(E_t)
&\le \sum_{B\in\mathcal{F}}\mu_\rho(5B)
\lesssim \sum_{B\in\mathcal{F}}\mu_\rho(\tfrac{1}{\tilde C}B)\\
&\lesssim \left(\frac{r}{t}\right)^{p/\alpha}\, \frac{1}{\mu_\rho(B_{0,0})^{(1-\alpha)/\alpha}}\,  \sum_{B\in\mathcal{F}}\,  
    \left(\int_{\tfrac{5\lambda a_2}{ \tilde Ca_1}\, B} g(x)^p\, d\mu_\rho(x) \right)^{1/\alpha}.
\end{align*}
As $0<\alpha<1$, it follows that
\begin{equation}
\label{eq:est3}
\mu_\rho(E_t)\lesssim \left(\frac{r}{t}\right)^{p/\alpha}\, \frac{1}{\mu_\rho(B_{0,0})^{(1-\alpha)/\alpha}}\, 
 \left(\int_{B_\rho(z,C_*r)} g(x)^p\, d\mu_\rho(x) \right)^{1/\alpha}.
\end{equation}
Here we used the fact that $\frac{5 \lambda a_2}{\tilde{C} a_1} \leq 1$  making the balls $\frac{5\lambda a_2}{\tilde{C}a_1}B$ disjoint and $C_*=\frac{8c_0}{(16C_{U,\rho})^2}+2C_{U,\rho}+1$ is obtained from the requirement that $\frac{5\lambda a_2}{\tilde{C}a_1}B\subset B_\rho(z,C_*r)$.
Notice that $C_*$  depends only on $C_U$, $C_A$ and $C_B$.

Recall from the Cavalieri principle~\cite[Proposition 6.24]{Fol} that
\[
\int_{B_{\rho}(z,r)}|u(x)-u_{B_{0,0}}|\, d\mu_\rho(x)=\int_0^\infty\mu_\rho(E_t)\, dt.
\]
Let $H>0$. For $0<t<H$ we estimate $\mu_\rho(E_t) \le \mu_\rho(B_{\rho}(z,r))$ and for $t\ge H$ we use the estimate~\eqref{eq:est3}.
This yields
\begin{align*}
\int_{B_{\rho}(z,r)}|&u(x)-u_{B_{0,0}}|\, d\mu_\rho(x)
\lesssim \mu_\rho(B_{\rho}(z,r))\, H\,
\\
&+\, \frac{r^{p/\alpha}}{\mu_\rho(B_{0,0})^{(1-\alpha)/\alpha}}\, \left(\int_{B_\rho(z,C_* r)}g(x)^p\, d\mu_\rho(x) \right)^{1/\alpha}\, \int_H^\infty t^{-p/\alpha}\, dt
\\
&= \mu_\rho(B_{\rho}(z,r))\, H\,
\\
&+\,\left[\frac{r^{p/\alpha}}{\mu_\rho(B_{0,0})^{(1-\alpha)/\alpha}}\, \left(\int_{B_\rho(z,C_* r)}g(x)^p\, d\mu_\rho(x) \right)^{1/\alpha} \frac{\alpha}{p-\alpha} \right]\, H^{1-p/\alpha}.
\end{align*}
With $A:=\mu_\rho(B_\rho(z,r))$ and
\begin{equation*}
E:=\left[\frac{r^{p/\alpha}}{\mu_\rho(B_{0,0})^{(1-\alpha)/\alpha}}\, \left(\int_{B_\rho(z,C_* r)}g(x)^p\, d\mu_\rho(x) \right)^{1/\alpha} \frac{\alpha}{p-\alpha} \right],
\end{equation*}
we see that the function $\psi(t)=A\, t+ E\, t^{1-p/\alpha}$ achieves its minimum at the value of $t$ for which $t^{p/\alpha}= \frac{E}{A} ( \frac{p}{\alpha} -1)$. Choosing $H$ to be this value of $t$, that is,
\begin{equation*}
H:=\frac{r}{\mu_\rho(B_{0,0})^{(1-\alpha)/p}}\, \frac{1}{\mu_\rho(B_\rho(z,r))^{\alpha/p}} \left(\int_{B_\rho(z,C_*r)}g(x)^p\, d\mu_\rho(x) \right)^{1/p},
\end{equation*}
we see that 
\[
\int_{B_\rho(z,r)}|u(x)-u_{B_{0,0}}|\, d\mu_\rho(x)
  \lesssim r\, 
  \frac{\mu_\rho(B_\rho(z,r))^{1-\alpha/p}}{\mu_\rho(B_{0,0})^{(1-\alpha)/p}}\, \left(\int_{B_\rho(z,C_*r)}g(x)^p\, d\mu_\rho(x) \right)^{1/p}.
\]
From the doubling property of $\mu_\rho$ it follows that $\mu_\rho(B_\rho(z,C_*r))\simeq \mu_\rho(B_\rho(z,r))$, and hence
\begin{align*}
\vint_{B_\rho(z,r)}|u(x)-u_{B_{0,0}}|\, d\mu_\rho(x)
&\lesssim r\, \frac{\mu_\rho(B_\rho(z,r))^{-\alpha/p}}{\mu_\rho(B_{0,0})^{(1-\alpha)/p}}\, \left(\int_{B_\rho(z,C_*r)}g(x)^p\, d\mu_\rho(x) \right)^{1/p}\\
\lesssim r\, &\frac{\mu_\rho(B_\rho(z,r))^{(1-\alpha)/p}}{\mu_\rho(B_{0,0})^{(1-\alpha)/p}}\, \left(\vint_{B_\rho(z,C_*r)}g(x)^p\, d\mu_\rho(x) \right)^{1/p}\\
&\lesssim r\, \left(\vint_{B_\rho(z,C_*r)}g(x)^p\, d\mu_\rho(x) \right)^{1/p},
\end{align*}
where we have used the fact that $\mu_\rho(B_\rho(z,r))/\mu_\rho(B_{0,0})\lesssim 1$.
This concludes the proof of the $p$-Poincar\'e inequality. 
\end{proof}

\noindent {\bf Addresses:} \\

	\vskip .2cm
	
	\noindent R.K.: Department of Mathematics and Systems Analysis, Aalto University, P.O. Box 11100, FI-00076 Aalto, Finland.
	\\
	\noindent E-mail:  R.K.: {\tt riikka.korte@aalto.fi}\\
	
	\vskip .2cm
	
	\noindent S.R.: Department of Mathematics and Systems Analysis, Aalto University, P.O. Box 11100, FI-00076 Aalto, Finland.
	\\
	\noindent E-mail:  S.R.: {\tt sari.rogovin@aalto.fi}\\
	
	\vskip .2cm
	
	\noindent N.S.: Department of Mathematical Sciences, P.O.~Box 210025, University of Cincinnati, Cincinnati, OH~45221-0025, U.S.A.\\
	\noindent E-mail:  N.S.: {\tt shanmun@uc.edu}\\

\noindent T.T.: Department of Mathematics and Systems Analysis, Aalto University, P.O. Box 11100, FI-00076 Aalto, Finland.
	\\
	\noindent E-mail:  T.T.: {\tt timo.i.takala@aalto.fi}\\
	
	\vskip .2cm

\end{document}